\documentclass[11pt]{amsart}
\usepackage{epsfig,url}
\usepackage{mathdots}

\newtheorem{theorem}{Theorem}[section]
\newtheorem{lemma}[theorem]{Lemma}

\newtheorem{corollary}[theorem]{Corollary}
\newtheorem{proposition}[theorem]{Proposition}

\theoremstyle{definition}

\newtheorem{note}[theorem]{Note}
\let\dotlessi=\i

\theoremstyle{remark}

\begin{document}

\title[Self-reciprocal functions and modular transformations]
{Self-reciprocal functions, powers of the {R}iemann zeta function and 
modular-type transformations}

\author{Atul Dixit}
\address{Department of Mathematics,
Tulane University, New Orleans, LA 70118}
\email{adixit@tulane.edu}

\author{Victor H. Moll}
\address{Department of Mathematics,
Tulane University, New Orleans, LA 70118}
\email{vhm@tulane.edu}


\date{\today}


\begin{abstract}
The classical transformation of Jacobi's theta function admits a simple proof
by producing an integral representation that yields this invariance apparent.
This idea seems to have first appeared in the work of S. Ramanujan. 
Several examples of this idea have been produced by Koshlyakov, Ferrar,
Guinand, Ramanujan and others. A unifying procedure to analyze these examples 
and natural generalizations is presented.
\end{abstract}

\maketitle

\newcommand{\ba}{\begin{eqnarray}}
\newcommand{\ea}{\end{eqnarray}}
\newcommand{\ift}{\int_{0}^{\infty}}
\newcommand{\nn}{\nonumber}
\newcommand{\no}{\noindent}
\newcommand{\lf}{\left\lfloor}
\newcommand{\rf}{\right\rfloor}
\newcommand{\realpart}{\mathop{\rm Re}\nolimits}
\newcommand{\imagpart}{\mathop{\rm Im}\nolimits}

\newcommand{\op}[1]{\ensuremath{\operatorname{#1}}}
\newcommand{\pFq}[5]{\ensuremath{{}_{#1}F_{#2} \left( \genfrac{}{}{0pt}{}{#3}
{#4} \bigg| {#5} \right)}}

\newtheorem{Definition}{\bf Definition}[section]
\newtheorem{Thm}[Definition]{\bf Theorem}
\newtheorem{Example}[Definition]{\bf Example}
\newtheorem{Lem}[Definition]{\bf Lemma}
\newtheorem{Cor}[Definition]{\bf Corollary}
\newtheorem{Prop}[Definition]{\bf Proposition}
\numberwithin{equation}{section}

\section{Introduction}
\label{sec-intro}

In his approach to the theory of elliptic functions, C.~G.~J.~Jacobi 
\cite{jacobi-1829a} introduced his classical theta function
\begin{equation}
\vartheta_{3}(x,\omega) = \sum_{n = -\infty}^{\infty} 
e^{2 \pi i n x + \pi i n^{2} \omega}
\end{equation}
\noindent
and three other similar functions $\vartheta_{1}, \, \vartheta_{2}, \, 
\vartheta_{4}$. These are entire functions of $x \in \mathbb{C}$, so they 
cannot be doubly-periodic, but every elliptic function can be written in 
terms of them. The transformations of the so-called \textit{null-values}
$\vartheta_{1}'(0,\omega), \, \vartheta_{j}(0,\omega)$ for $2 \leq j \leq 4$
under the modular group $PSL(2, \mathbb{Z})$ are of intrinsic 
interest.  Jacobi proved that the transformation of $\vartheta_{3}(0,it)$ for Re $t>0$,
yields the pretty identity
\begin{equation}
\sum_{n=-\infty}^{\infty} e^{-\pi n^{2}t} = \frac{1}{\sqrt{t}}
\sum_{n=-\infty}^{\infty} e^{-\pi n^{2}/t}.
\end{equation}
\noindent
The reader will find details in \cite[Chapter 3]{mckean-1997a}. This identity
may be written in a more symmetric form as
\begin{equation}
\sqrt{\alpha} \left( \frac{1}{2 \alpha} - \sum_{n=1}^{\infty} 
e^{-\pi \alpha^{2} n^{2}} \right) = 
\sqrt{\beta} \left( \frac{1}{2 \beta} - \sum_{n=1}^{\infty} 
e^{-\pi \beta^{2} n^{2}} \right),
\label{jacobi-1}
\end{equation}
\noindent
with $\alpha, \, \beta > 0$ and $\alpha \beta = 1$. These type of 
identities are called modular-type transformations.
%
A procedure to establish the identity \eqref{jacobi-1} is to produce an
integral representation of one of the sides that is invariant under
$\alpha \to 1/\alpha$. To see this, define 
\begin{equation}
\Xi(t) = \xi \left( \tfrac{1}{2} + i t \right),
\end{equation}
\noindent
where $\xi$ is the Riemann $\xi$-function
\begin{equation}
\xi(s) = \tfrac{1}{2}s(s-1)\pi^{-s/2} \Gamma \left( \tfrac{s}{2} \right) 
 \, \zeta(s),
\label{def-xi}
\end{equation}
\noindent 
with 
\begin{equation}
\Gamma(s) = \int_{0}^{\infty} t^{s-1} e^{-t} \, dt \text{ and }
\zeta(s) = \sum_{n=1}^{\infty} \frac{1}{n^{s}},
\label{def-gamma-00}
\end{equation}
\noindent
the gamma and Riemann zeta function, respectively. The definition of $\Gamma(s)$ in 
\eqref{def-gamma-00} is valid for $\realpart{s} > 0$ and that for $\zeta(s)$ is valid for $\realpart{s}>1$. These functions
are then extended to $\mathbb{C}$ by meromorphic continuation, with poles  
at $s = 0, \, -1, \, -2, \ldots$ for $\Gamma(s)$, and at $s=1$ for $\zeta(s)$. The integral evaluation 
\cite[p. 36]{paris-2001a}
\begin{equation}
\frac{2}{\pi} \int_{0}^{\infty} \frac{\Xi(t/2)}{1+t^{2}} 
\cos( \tfrac{1}{2} t \log \alpha) \, dt = 
\sqrt{\alpha} \left( \frac{1}{2 \alpha} - \sum_{n=1}^{\infty} 
e^{-\pi \alpha^{2} n^{2}} \right) 
\label{integ-11}
\end{equation}
\noindent
and the obvious invariance of the left-hand side under $\alpha \to 1/\alpha$
gives \eqref{jacobi-1}.

\smallskip

The goal of this paper is to study, in a unified manner, a variety of 
integrals of the form 
\begin{equation}
\label{int-Igen}
I(f,z; \alpha) = \int_{0}^{\infty} f \left(z, \tfrac{t}{2} \right) 
\Xi \left( \tfrac{t  - i z}{2} \right) \, 
\Xi \left( \tfrac{t  + i z}{2} \right) \, 
\cos \left( \tfrac{1}{2} t \, \log \alpha 
\right) \, dt,
\end{equation}
\noindent
where $f(z,t)$ is an even function of $t$ of the form 
\begin{equation}
f(z, t) = \phi(z, i t) \phi(z, - i t ),\label{factorization-0}
\end{equation}
with $\phi$ analytic as a function of $t \in \mathbb{R}$ and 
$z \in \mathbb{C}$. The integral \eqref{int-Igen} extends
\begin{equation}
\label{int-I}
I(f; \alpha) = \int_{0}^{\infty} f \left( \tfrac{t}{2} \right) 
\Xi\left( \tfrac{t}{2} \right)
\cos \left( \tfrac{1}{2} t \, \log \alpha 
\right) \, dt,
\end{equation}
\noindent
with
\begin{equation}
f(t) = \varphi(i t) \varphi(- i t ),
\label{factorization-1}
\end{equation}
and $\varphi$ is analytic in $t$. 
Particular examples of this latter integral were studied by Ramanujan, Hardy, 
Koshlyakov and Ferrar (see also \cite[p. 35]{titchmarsh-1986a}), 
whereas Ramanujan \cite{ramanujan-1915a} was the first person to study an integral of the type in \eqref{int-Igen}.
\noindent
It is clear that \eqref{int-Igen} and \eqref{int-I}
are invariant under $\alpha \to 1/\alpha$. An alternative 
expression (series or integral) of 
$I(f,\alpha)$ then yields identities similar to \eqref{jacobi-1}. These 
are called \textit{modular-type transformations}. It is easy to extend
these identities by analytic continuation to 
$\alpha, \, \beta \in \Omega \subset \mathbb{C}$, with 
$\mathbb{R} \subset \Omega$. The 
classical example in \eqref{jacobi-1} corresponds to taking $f(t) = 1/(1+4t^{2})$ in \eqref{int-I}. 

The identities obtained here have the form 
\begin{equation}
\label{method-1}
F(z, \alpha) = F(z, \beta) = \int_{0}^{\infty} f \left(z, \tfrac{t}{2} \right) 
\Xi \left( \tfrac{t  - i z}{2} \right) \, 
\Xi \left( \tfrac{t  + i z}{2} \right) \, 
\cos \left( \tfrac{1}{2} t \, \log \alpha 
\right) \, dt,
\end{equation}
\noindent
for a suitably chosen function $f$. The parameters $\alpha$ and 
$\beta$ are positive and satisfy $\alpha \beta  = 1$. Naturally, if the 
identity 
\begin{equation}
\label{method-2}
F(z, \alpha) = \int_{0}^{\infty} G(z, t) \cos \left( \tfrac{1}{2} t 
\, \log \alpha \right) \, dt,
\end{equation}
\noindent
has been established, it follows immediately that $F(z, \alpha) = F(z, \beta)$. 

\smallskip

Section \ref{sec-illustrative} introduces a technique for studying integrals 
in \eqref{int-I} through an example 
containing $K_{0}(x)$, the modified Bessel function of order $0$. Section \ref{sec-ramanujan} 
describes an example found in Ramanujan's work and links it to 
another of his integrals. This illustrates the fundamental idea of this paper. 
In Section  \ref{sec-param-intro}, some classical formulas are
generalized by the introduction of a new 
parameter $z$. This section states the results with the proofs
presented in Section \ref{sec-proofs}. Generalizations of some results of 
Koshlyakov, particularly dealing with his function $\Omega(x)$ 
(see (\ref{koshomega}) below), are stated in Section \ref{sec-omega}. 
Section \ref{sec-advantages} discusses advantages of our methods over 
those of A.P.~Guinand and C.~Nasim. Finally, the last section describes future 
directions of the work discussed here.

\section{An illustrative example}
\label{sec-illustrative}

This section presents a general technique to evaluate integrals of the 
type \eqref{int-I}. This is illustrated with an example 
established by N.~S.~Koshlyakov 
\cite{koshlyakov-1929a}. The proof given here follows \cite{dixit-2010b} and 
\cite{ferrar-1936a}. 

The function 
\begin{equation}
f(t) = \frac{4\Xi(t)}{(\tfrac{1}{4}+t^{2})^{2}}
\label{factor-2}
\end{equation}
\noindent
admits a factorization in the form \eqref{factorization-1} with 
\begin{equation}\label{factor-3}
\varphi(s) = \frac{2 \sqrt{ \xi \left( \tfrac{1}{2} - s \right)}}
{ \left( \tfrac{1}{2} + s \right) \left( \tfrac{1}{2} - s \right)}.
\end{equation}
\noindent
It is known \cite[p.~35]{titchmarsh-1986a} that
\begin{equation}
\int_{0}^{\infty} f \left( \tfrac{t}{2} \right) 
\Xi \left( \tfrac{t}{2} \right) 
\, \cos \left(\tfrac{1}{2}t\log\alpha\right) \, dt = 
\frac{1}{i \sqrt{y}}
\int_{\tfrac{1}{2} - i \infty}^{\tfrac{1}{2} + i \infty} 
\varphi \left( s - \tfrac{1}{2} \right) 
\varphi \left( \tfrac{1}{2} -s \right) 
\xi(s) \alpha^{s} \, ds
\end{equation}
\noindent
Using \eqref{factor-2} and \eqref{factor-3}, this yields 
\begin{equation}
\label{line-integral1}
\int_{0}^{\infty} \frac{ 64 \, \Xi^{2} \left( \tfrac{t}{2} \right) \, 
\cos( \tfrac{1}{2}t \log \alpha ) }{(1+t^{2})^{2}} \, dt =
 \frac{1}{i \sqrt{y}}
\int_{\tfrac{1}{2} - i \infty}^{\tfrac{1}{2} + i \infty} 
\Gamma^{2}\left( \tfrac{s}{2} \right) \zeta^{2}(s) \left( \frac{\alpha}{\pi} 
\right)^{s} \, ds.
\end{equation}
\noindent
To evaluate the contour integral on the right-hand side, square the 
functional equation for $\zeta(s)$
\begin{equation}
\pi^{-s/2} \Gamma \left( \frac{s}{2} \right) \zeta(s) = 
\pi^{-(1-s)/2} \Gamma \left( \frac{1-s}{2} \right) \zeta(1-s),
\label{functional-zeta}
\end{equation}
\noindent
and recall the Dirichlet series 
\begin{equation}
\zeta^{2}(s) = \sum_{n=1}^{\infty} \frac{d(n)}{n^{s}},
\label{dirichlet-1}
\end{equation}
\noindent
where $d(n)$ is the number of divisors of $n$; 
see \cite[page 229]{apostol-1976a}. The expansion \eqref{dirichlet-1} is valid 
for $\realpart{s} > 1$, so it is necessary to move the line of integration in 
\eqref{line-integral1} to $\realpart{s} = 1+ \delta$, for some 
$\delta > 0$. This process captures a pole of the integrand at $s=1$, with 
residue 
\begin{equation}
\lim\limits_{s \to 1} \frac{d}{ds} 
\left( (s-1)^{2} \left( \frac{\alpha}{\pi} \right)^{s} 
\Gamma^{2} \left( \frac{s}{2} \right) \zeta^{2}(s) \right) = 
\alpha \left( \gamma - \log \left( \frac{4 \pi}{\alpha} \right) \right),
\end{equation}
\noindent
where $\gamma = - \Gamma'(1)$ is the Euler's constant. Then 
\begin{equation*}
\int_{1+ \delta - i \infty}^{1+ \delta + i \infty} 
\Gamma^{2} \left( \frac{s}{2} \right) \zeta^{2}(s) \left( \frac{y}{\pi} 
\right)^{s} ds = 
\sum_{m=1}^{\infty} d(m) 
\int_{1+ \delta - i \infty}^{1+ \delta + i \infty} 
\Gamma^{2} \left( \frac{s}{2} \right) \left( \frac{\pi m}{y} \right)^{-s} ds.
\end{equation*}
\noindent
The integral on the right-hand side appears in 
\cite[p.115, formula 11.1]{oberhettinger-1974a}
\begin{equation}
\frac{1}{2 \pi i} \int_{c - i \infty}^{c + i \infty} 
2^{s-2} a^{-s} \Gamma\left( \frac{s}{2} - \frac{\nu}{2} \right) 
\Gamma\left( \frac{s}{2} + \frac{\nu}{2} \right) x^{-s} \, ds = 
K_{\nu}(ax),\label{mellin-KBessel}
\end{equation}
\noindent
valid for Re $s>\pm $ Re $\nu$ and $a>0$. Here $K_{\nu}(w)$ is the modified Bessel function of order $\nu$, defined by \cite[p.~928, formula \text{8.485}]{gradshteyn-2007a}
\begin{equation}
K_{\nu}(w) = \frac{\pi}{2} \, 
\frac{ \left( I_{-\nu}(w) - I_{\nu}(w) \right)}{\sin \pi \nu} 
\end{equation}
\noindent
where \cite[p.~911, formula 8.406, nos. 1-2]{gradshteyn-2007a}
\begin{equation}\label{defbesi}
I_{\nu}(w)=
\begin{cases}
e^{-\frac{\pi}{2}\nu i}J_{\nu}(e^{\frac{\pi i}{2}}w) & \text{if $-\pi<$ arg $w\leq\frac{\pi}{2}$,}\\
e^{\frac{3}{2}\pi\nu i}J_{\nu}(e^{-\frac{3}{2}\pi i}w) & \text{if $\frac{\pi}{2}<$ arg $w\leq \pi$,}
\end{cases}
\end{equation}
and \cite[p.~910, formula 8.402]{gradshteyn-2007a}
\begin{equation}\label{defbesj}
J_{\nu}(w):=\left(\frac{w}{2}\right)^{\nu}\sum_{n=0}^{\infty}\frac{(-w^2/4)^n}{\Gamma(\nu+1+n)n!}
\end{equation}
is the Bessel function of the first kind. The 
result is an integral of type \eqref{int-I}, stated here for $\nu=0$.

\begin{theorem}
Let $d(n)$ be the number of divisors of $n \in \mathbb{N}, \, 
\gamma$ the Euler constant and
$K_0(w)$ the modified Bessel function of order $0$. Then, for $\alpha, \beta > 0, \alpha\beta=1$,
\begin{eqnarray}
\label{int-kosh1}\\
- \frac{32}{\pi} \int_{0}^{\infty} \frac{ \left( \Xi \left( \tfrac{t}{2}
\right) \right)^{2}}{(1+t^{2})^{2}} \cos( \tfrac{1}{2}t \log \alpha ) \, dt 
& = &   \sqrt{\alpha} \left( \frac{\gamma - \log( 4 \pi \alpha)}{\alpha} - 
4 \sum_{n=1}^{\infty} d(n) K_{0}(2 \pi n \alpha) \right) \nonumber \\
& = & \sqrt{\beta} \left( \frac{\gamma - \log( 4 \pi \beta)}{\beta} - 
4 \sum_{n=1}^{\infty} d(n) K_{0}(2 \pi n \beta) \right).
\nonumber
\end{eqnarray}
\end{theorem}

\section{An example from {R}amanujan}
\label{sec-ramanujan}

The success of the method described in Section \ref{sec-illustrative} 
depends on the appropriate choice of the function $\varphi(t)$ in 
\eqref{factorization-1} and the ability to evaluate, or at least transform, 
the resulting integral $I(f; \alpha)$. This section presents a second 
example that illustrates this point of view. 

Take 
\begin{equation}
\varphi(t) = \Gamma \left( \frac{2t-1}{4} \right)
\end{equation}
\noindent
and 
\begin{equation}
f(t) = 
\Gamma \left( \frac{2it-1}{4} \right)
\Gamma \left( \frac{-2it-1}{4} \right) = 
\left| \Gamma \left( \frac{2it-1}{4} \right) \right|^{2},
\end{equation}
\noindent
in view of $\overline{\Gamma(z)} = \Gamma( \overline{z} )$. The 
integral \eqref{int-I} becomes 
\begin{eqnarray}
\quad I(f;\alpha) & = & 
\int_{0}^{\infty} 
\left| \Gamma \left( \frac{it-1}{4} \right) \right|^{2}
\Xi \left( \frac{t}{2} \right) \cos \left( \frac{1}{2} t \log \alpha \right) dt
\label{rama-1} \\
& = & 
\int_{0}^{\infty} 
\Gamma \left( \frac{it-1}{4} \right) 
\Gamma \left( \frac{-it-1}{4} \right) 
\Xi \left( \frac{t}{2} \right) \cos \left( \frac{1}{2} t \log \alpha \right) dt.
\nonumber
\end{eqnarray}
\noindent
The functional equation $\Gamma(z+1) = z \Gamma(z)$, with $z = (it -1)/4$, now
gives a new representation of \eqref{rama-1}:
\begin{equation}
\label{rama-2}
I(f; \alpha) = 16 
\int_{0}^{\infty} 
\Gamma \left( \frac{3 + it}{4} \right) 
\Gamma \left( \frac{3-it}{4} \right) 
\frac{\Xi \left( \frac{t}{2} \right)}{1+t^{2}} 
 \cos \left( \frac{1}{2} t \log \alpha \right) dt.
\end{equation}
This integral was transformed by S.~Ramanujan 
\cite[Equation (13)]{ramanujan-1915a}. 

\begin{theorem}[Ramanujan]\label{ramanujan}
The identity 
\begin{multline}
\int_{0}^{\infty} 
\Gamma \left( \frac{3 + it}{4} \right) 
\Gamma \left( \frac{3-it}{4} \right) 
\frac{\Xi \left( \frac{t}{2} \right)}{1+t^{2}} 
\cos \left( \frac{1}{2} t \log \alpha \right) dt = 
\label{rama-44} \\
\pi^{5/2} \alpha^{3/2} \int_{0}^{\infty} 
xe^{- \pi \alpha^{2} x^{2}} \left( \frac{1}{2 \pi x} - 
\frac{1}{e^{2 \pi x } -1 } \right) \, dx
\end{multline}
\noindent
holds for $\alpha>0$.
\end{theorem}

This evaluation generates a modular-type transformation, which 
naturally leads to a beautiful identity among definite integrals.

\begin{corollary}\label{ramanujancor}
Let $\alpha, \, \beta > 0$ with $\alpha \beta =1$. Then 
\begin{multline*}
\alpha^{3/2} \int_{0}^{\infty} 
xe^{- \pi \alpha^{2} x^{2}} \left( \frac{1}{2 \pi x} - 
\frac{1}{e^{2 \pi x } -1 } \right) \, dx = \\
\beta^{3/2} \int_{0}^{\infty} 
xe^{- \pi \beta^{2} x^{2}} \left( \frac{1}{2 \pi x} - 
\frac{1}{e^{2 \pi x } -1 } \right) \, dx.
\end{multline*}
\end{corollary}

\begin{note}
Koshlyakov example \eqref{int-kosh1} is obtained by squaring
$\Xi(t/2)/(1+t^{2})$, which is part of the integrand 
in \eqref{integ-11} appearing in the proof of the classical 
theta function identity. Thus is it natural to consider the integral 
\begin{equation}
\int_{0}^{\infty} 
\Gamma \left( \frac{3 + it}{4} \right)
\Gamma \left( \frac{3 - it}{4} \right)
\left( \frac{\Xi \left( \tfrac{t}{2} \right)}{1+t^{2}} \right)^{2} 
\, \cos \left( \tfrac{1}{2} t \log \alpha \right) \,dt 
\end{equation}
\noindent
as a variation of \eqref{rama-44}. Up to a constant factor, this may be 
expressed as 
\begin{equation}
\int_{0}^{\infty} 
\Gamma \left( \frac{-1 + it}{4} \right)
\Gamma \left( \frac{-1 - it}{4} \right)
\frac{\Xi^{2} \left( \tfrac{t}{2} \right)}{1+t^{2}} 
\, \cos \left( \tfrac{1}{2} t \log \alpha \right) \,dt, 
\label{ramanujan-99}
\end{equation}
\noindent
which is exactly the integral presented by S. Ramanujan at the end of his
paper \cite[Equation (22)]{ramanujan-1915a}. Thus the idea of 
squaring the functional equation of $\zeta(s)$ to produce new 
transformations is implicit in the work of Ramanujan, much before Koshlyakov. In his Lost Notebook 
\cite{ramanujan-1988a}, Ramanujan gives the following beautiful modular-type 
transformation resulting from this integral.

\begin{theorem}
\label{thm-3.4}
Let 
\begin{equation}
\lambda(x) = \psi(x) + \frac{1}{2x} - \log x,
\end{equation}
\noindent
where 
\begin{equation}
\psi(x) = \frac{\Gamma'(x)}{\Gamma(x)} = - \gamma - 
\sum_{m=0}^{\infty} \left( \frac{1}{m+x} - \frac{1}{m+1} \right)
\end{equation}
\noindent
is the logarithmic derivative of the Gamma function. If $\alpha$ and 
$\beta $ are positive numbers such that $\alpha \beta =1$, then 
\begin{multline*}
\sqrt{\alpha} \left( \frac{\gamma - \log( 2 \pi \alpha)}{2 \alpha} + 
\sum_{k=1}^{\infty} \lambda( k \alpha) \right) = 
\sqrt{\beta} \left( \frac{\gamma - \log( 2 \pi \beta)}{2 \beta} + 
\sum_{k=1}^{\infty} \lambda( k \beta) \right)  \\
= - \frac{1}{\pi^{3/2}} 
\int_{0}^{\infty} \left| \Gamma \left( \frac{-1+it}{4} \right) \right|^{2}\Xi^{2} \left( \frac{t}{2} \right) \, 
\frac{\cos \left( \tfrac{1}{2} t \log \alpha \right) \, dt}{1+t^{2}}.
\end{multline*}
\end{theorem}
\end{note}

\section{Some parametric generalizations. Statements of the results}
\label{sec-param-intro}

This section contains several new modular-type transformations. The results 
are stated in this section with their proofs
postponed to Section \ref{sec-proofs}. 

There are (at least) two approaches towards obtaining these 
transformations. One is presented in Theorems \ref{production-1} and 
\ref{production-2} and the corollaries following them. The other method 
involves the evaluation of integrals involving the Riemann $\Xi$-function, 
illustrated by \eqref{ramanujan-98} below. Given a modular-type transformation,
the use of Parseval's identity \eqref{parseval-1} produces an integral 
involving the Riemann $\Xi$-function. Conversely, having a representation for 
an aforementioned integral involving the Riemann $\Xi$-function, obtained 
by residue calculus and Mellin transforms, the obvious invariance of 
these integrals under $\alpha \mapsto 1/\alpha$, provides a 
modular-type transformation. Examples of these two methods are presented in 
the proofs of the main results.

Implicit in both methods is the idea of squaring the functional equation 
of $\zeta(s)$. In the second method, this is reflected in the term 
$\Xi\left( \frac{t}{2} \right)/(1+t^2)$, present in the integrands. This 
is generalized by the inclusion a new parameter $z$. 

In his work related to \eqref{ramanujan-99}, Ramanujan
\cite{ramanujan-1915a} considered the generalization
\begin{multline}
\int_{0}^{\infty} 
\Gamma \left( \frac{z-1+ i t }{4} \right)
\Gamma \left( \frac{z-1- i t }{4} \right) \\
\times \Xi \left( \frac{t + i z}{2} \right) 
\Xi \left( \frac{t - i z}{2} \right) 
\frac{\cos( \tfrac{1}{2} t \log \alpha) \, dt}{(t^{2} + (z+1)^{2}}
\label{ramanujan-98}
\end{multline}
of \eqref{ramanujan-99}. He provided 
alternative integral representations valid in different regions 
in the complex plane. Modular-type transformations for the above integral, which involve the Hurwitz 
zeta function, appear in \cite{dixit-2011a, dixit-2011b}. It should be 
mentioned here that Ramanujan had not only discovered Koshlyakov's formula
\eqref{int-kosh1} about $10$ years before, but had also generalized it. Details
appear in \cite{berndt-2008a}. The generalization presented below,
was later rediscovered by A. P. Guinand \cite{guinand-1955a} and is rephrased 
in a symmetric form in the theorem below. 
\begin{theorem}
Let $K_{\nu}(s)$ denote the modified {B}essel function of order $\nu$ and 
let $\begin{displaystyle} \sigma_{s}(n) = \sum_{d | n} d^{s} 
\end{displaystyle}$. For 
$-1 < \realpart{z} < 1$ define 
\begin{equation}
\omega(z, \alpha) = \left( \frac{\alpha}{\pi} \right)^{z/2} 
\Gamma \left( \frac{z}{2} \right) \zeta(z).
\end{equation}
\noindent
Then, if $\alpha, \, \beta > 0$ with $\alpha \beta  =1$, 
\begin{multline}\label{guinandf1}
\frac{1}{\sqrt{\alpha}} \left( \omega(z, \alpha) + \omega(-z,\alpha) - 
4 \alpha \sum_{n=1}^{\infty} \sigma_{-z}(n) n^{z/2} K_{z/2}(2 n \pi \alpha) 
\right) = \\
\frac{1}{\sqrt{\beta}} \left( \omega(z, \beta) + \omega(-z,\beta) - 
4 \beta \sum_{n=1}^{\infty} \sigma_{-z}(n) n^{z/2} K_{z/2}(2 n \pi \beta) 
\right).
\end{multline}
\end{theorem}

\begin{note}\label{int-rgf}
The symmetry in $\alpha$ and $\beta$ suggests the existence of an integral 
involving the Riemann $\Xi$-function similar to \eqref{ramanujan-98}, which 
generalizes that giving rise to Koshlyakov's formula \eqref{int-kosh1}. This 
integral, found in \cite{dixit-2011b}, is 
\begin{equation}
- \frac{32}{\pi} \int_{0}^{\infty} 
\Xi \left( \frac{t + iz}{2} \right) 
\Xi \left( \frac{t - iz}{2} \right) 
\frac{\cos \left( \tfrac{1}{2} t \log \alpha \right) \, dt}
{(t^{2} + (z+1)^{2})(t^{2}+(z-1)^{2})}. 
\label{atul-00}
\end{equation}
\noindent
The term $\begin{displaystyle} \Xi^{2}(t/2)/(1+t^{2})^{2} 
\end{displaystyle}$ in \eqref{int-kosh1}, which resulted from squaring part
of the integrand in \eqref{integ-11}, is now generalized to 
\begin{equation*}
\frac{ 
\Xi \left( \frac{t + i z}{2} \right) 
\Xi \left( \frac{t - i z}{2} \right) }
{ (t^{2} + (z+1)^{2})(t^{2}+ (z-1)^{2})}
\end{equation*}
\noindent 
in \eqref{atul-00}. The integral in \eqref{ramanujan-98} can also be 
rewritten to contain the above expression. To the best of our 
knowledge, \eqref{atul-00} is the only other 
integral of this type, besides Ramanujan's  \eqref{ramanujan-98}, that 
has been studied up to now. Several new integrals of this type are 
presented next. 
\end{note}

\begin{note}
As mentioned above, Koshlyakov made use of the idea of squaring the 
functional equation for $\zeta(s)$ in order to obtain some new transformations through the existing ones.
These include Hardy's formula \cite{hardy-1915a}, rephrased in a compact form
given by Koshlyakov \cite[Equations (14), (20)]{koshlyakov-1937a}:
\begin{align}\label{hardyf}
\sqrt{\alpha} \int_{0}^{\infty} e^{-\pi \alpha^{2} x^{2}} 
\left( \psi(x+1) - \log x \right) \, dx &= \sqrt{\beta} \int_{0}^{\infty} e^{-\pi \beta^{2} x^{2}} 
\left( \psi(x+1) - \log x \right) \, dx\\
&=2 \int_{0}^{\infty} \frac{\Xi(t/2)}{1+t^{2}} \, 
\frac{\cos \left( \tfrac{1}{2}t \log \alpha \right)}{\cosh \tfrac{1}{2} 
\pi t} \,  dt. 
\nonumber
\end{align}
\noindent 
\end{note}

Koshlyakov \cite[Equations (36), (40)]{koshlyakov-1937a} squared the term 
$\Xi(t/2)/(1+t^{2})$ in the integral on the extreme right above and obtained 
the following result\footnote{In equation (40) in Koshlyakov's paper, there
is a minus sign missing in front on the right-hand side.}. 

\begin{theorem}
Define 
\begin{equation}
\Lambda(x) = \frac{\pi^{2}}{6} + \gamma^{2} - 2 \gamma_{1} + 2 \gamma \log x 
+ \frac{1}{2} \log^{2}x + \sum_{n=1}^{\infty} d(n) 
\left( \frac{1}{x+n} - \frac{1}{n} \right), 
\end{equation}
\noindent
where $\gamma_{1}$ is the Stieltjes constant defined by
\begin{equation*}
\gamma_1=\lim_{m\to\infty}\left(\sum_{k=1}^{m}\frac{\log k}{k}-\frac{(\log m)^2}{2}\right).
\end{equation*}
Then, for $\alpha, \beta > 0, \alpha\beta=1$, 
\begin{align}
\sqrt{\alpha} \int_{0}^{\infty} K_{0}(2 \pi \alpha x) \Lambda(x) \, dx &= \sqrt{\beta} \int_{0}^{\infty} K_{0}(2 \pi \beta x) \Lambda(x) \, dx\nonumber\\
&=8 \int_{0}^{\infty} \frac{ \left( \Xi \left( \tfrac{t}{2} \right) \right)^{2}}
{(1+t^{2})^{2}} \,  
\frac{\cos \left( \tfrac{1}{2} t \log \alpha \right)}{\cosh \left( \tfrac{1}{2}
\pi t \right)} \, dt. \label{kosh1937}
\end{align}
\noindent
\end{theorem}

In his work, Koshlyakov did not consider one variable generalizations of the 
type in \eqref{ramanujan-98} and \eqref{atul-00} of any of his formulas. 
The next result presents a generalization of \eqref{kosh1937} where 
$\Xi^{2}(t/2)/(1+t^{2})^{2}$ is generalized to 
\begin{equation*}
\frac{ 
\Xi \left( \frac{t + i z}{2} \right) 
\Xi \left( \frac{t - i z}{2} \right) }
{ (t^{2} + (z+1)^{2})(t^{2}+ (z-1)^{2})}.
\end{equation*}
Furthermore, the reciprocal of $\cosh\tfrac{1}{2}\pi t$ is replaced by a 
product of four gamma functions. 

\begin{theorem}
\label{thm-45}
Assume $-1 < \realpart{z} < 1$ and let $\gamma, \, \gamma_{1}$ and 
$K_{\nu}(z)$ be as before. Define 
\begin{multline}\label{lambda-def}
\Lambda(x,z) = x^{z/2} \Gamma(1+z) \left\{ 
\frac{x^{-z}}{-z} \zeta(1-z) + (2 \gamma + \log x + \psi(1+z) ) \zeta(1+z) 
\right. \\
\left. + \zeta'(1+z) + \sum_{n=1}^{\infty} \sigma_{-z}(n) 
\left( \frac{n^{z}}{(n+x)^{z+1}} - \frac{1}{n} \right) \right\}.
\end{multline}
\noindent
Then, for $\alpha,\beta > 0$ and $\alpha\beta=1$, 
\begin{multline}
\label{formula-121}
\sqrt{\alpha} \int_{0}^{\infty} K_{z/2}(2 \pi \alpha x) \Lambda(x,z) dx = \sqrt{\beta} \int_{0}^{\infty} K_{z/2}(2 \pi \beta x) \Lambda(x,z) dx\\
=\frac{2^{z+2}}{\pi^{2}} \int_{0}^{\infty} 
\Gamma \left( \frac{z+3+it}{4} \right) 
\Gamma \left( \frac{z+3-it}{4} \right) 
\Gamma \left( \frac{z+1+it}{4} \right)  \\
\Gamma \left( \frac{z+1-it}{4} \right) 
\Xi \left( \frac{t+iz}{2} \right) 
\Xi \left( \frac{t-iz}{2} \right) 
\frac{ \cos \left( \tfrac{1}{2} t \log \alpha \right) \, dt}
{(t^{2}+ (z+1)^{2})(t^{2} + (z-1)^{2})}. 
\end{multline}
\end{theorem}


\begin{theorem}
\label{thm-46}
Assume $-1 < \realpart{z} < 1$. Let 
\begin{equation}
\mu(x,z) = \frac{\Gamma( 1 + z) \zeta(1+z) x^{-1 -z/2}}{(2 \pi)^{2+z}} 
\end{equation}
\noindent
and define
\begin{equation}
\Phi(x,z)  =  2 \sum_{n=1}^{\infty} \sigma_{-z}(n) n^{z/2} K_{z}(
4 \pi \sqrt{nx} ) - \mu(x,z) - \mu(x,-z).
\label{def-Phi}
\end{equation}
\noindent
Then, for $\alpha,\beta > 0$ and $\alpha\beta=1$, 
\begin{multline}
\sqrt{\alpha^{3}} \int_{0}^{\infty} x K_{z/2}(2 \pi \alpha x) \Phi(x,z) dx =\sqrt{\beta^{3}} \int_{0}^{\infty} x K_{z/2}( 2 \pi \beta x) 
\Phi(x,z) dx\\
= \frac{2}{\pi^{4}} \int_{0}^{\infty} 
\Gamma \left( \frac{z+3+it}{4} \right) 
\Gamma \left( \frac{z+3-it}{4} \right) 
\Gamma \left( \frac{-z+3+it}{4} \right)  \\
\times \Gamma \left( \frac{-z+3-it}{4} \right) 
\Xi \left( \frac{t + i z}{2} \right) 
\Xi \left( \frac{t - i z}{2} \right) 
\frac{\cos \left( \tfrac{1}{2} t \log \alpha \right) \, dt}
{(t^{2} + (z+1)^{2})(t^{2} + (z-1)^{2})}. \label{thm-46id}
\end{multline}
\end{theorem}

The series $\sum_{n=1}^{\infty}\sigma_{-z}(n)n^{\frac{z}{2}}K_{z}(4\pi\sqrt{nx})$, along with some of its special cases, is treated at length in \cite{cohen-2010a}. The special case $z=0$ of the above theorem is interesting enough to be singled out. 

\begin{corollary}\label{cor-47}
Let $\alpha,\beta > 0$ and $\alpha\beta=1$. Then 
\begin{multline}
\sqrt{\alpha^{3}} 
\int_{0}^{\infty} x K_{0}(2 \pi \alpha x) 
\left( 2 \sum_{n=1}^{\infty} d(n) K_{0}(4 \pi \sqrt{nx}) + 
\frac{\log( 4 \pi^{2} x)}{4 \pi^{2}x} \right) dx = 
\label{corollary-47} \\
=\sqrt{\beta^{3}}\int_{0}^{\infty} x K_{0}(2 \pi \beta x) 
\left( 2 \sum_{n=1}^{\infty} d(n) K_{0}(4 \pi \sqrt{nx}) + 
\frac{\log( 4 \pi^{2} x)}{4 \pi^{2}x} \right) dx\\
=\frac{1}{128\pi^{4}} \int_{0}^{\infty} 
\Gamma^{2} \left( \frac{-1-it}{4} \right) 
\Gamma^{2} \left( \frac{-1+it}{4} \right) 
\Xi^{2} \left( \frac{t}{2} \right) \cos \left( \tfrac{1}{2} t \log \alpha 
\right) dt. 
\end{multline}
\end{corollary}
Using Vorono\"{\dotlessi}'s identity \cite[Equations (5), (6)]
{voronoi-1904a}
\begin{equation}\label{lostomega}
2\sum_{n=1}^{\infty}d(n)K_{0}(4\pi\sqrt{nx})=\frac{a}{\pi^2}\sum_{n=1}^{\infty}\frac{d(n)\log(x/n)}{x^2-n^2}-\frac{\gamma}{2}-\left(\frac{1}{4}+\frac{1}{4\pi^2x}\right)\log x-\frac{\log 2\pi}{2\pi^2x},
\end{equation}
the above transformation is written in an equivalent different form. It 
should be pointed out that these identities also appear in 
Ramanujan's Lost Notebook \cite[p.~254]{ramanujan-1988a} 
(see also \cite[Equation (4.1)]{berndt-2008a}).

This is now rephrased into yet another form. This provides an interesting 
modular-type transformation between two double integrals as shown below.
\begin{theorem}\label{jbesstra}
Let $J_{\nu}(x)$ denote the Bessel function of the first kind of order $\nu$. For $\alpha, \beta>0$, $\alpha\beta=1$, we have
\begin{align}\label{corollary-47r}
&\sqrt{\alpha}\int_{0}^{\infty}\int_{0}^{\infty}\frac{y}{(y^2+t^2)^{3/2}}\left(J_{0}(2\alpha y)+\frac{4\pi t}{e^{2\pi t}-1}\left(\frac{1}{e^{2\pi\alpha y}-1}-\frac{1}{2\pi\alpha y}\right)\right)\, dy\, dt\nonumber\\
&=\sqrt{\beta}\int_{0}^{\infty}\int_{0}^{\infty}\frac{y}{(y^2+t^2)^{3/2}}\left(J_{0}(2\beta y)+\frac{4\pi t}{e^{2\pi t}-1}\left(\frac{1}{e^{2\pi\beta y}-1}-\frac{1}{2\pi\beta y}\right)\right)\, dy\, dt\nonumber\\
&=\frac{1}{8\pi^2}\int_{0}^{\infty}\Gamma^{2}\left(\frac{-1-it}{4}\right)\Gamma^{2}\left(\frac{-1+it}{4}\right)\Xi^{2}\left(\frac{t}{2}\right)\cos\left( \frac{1}{2}t\log\alpha\right)\, dt.
\end{align}
\end{theorem}


\smallskip 

An identity of the same type as \eqref{rama-44} was established by W. L. Ferrar 
\cite{ferrar-1936a}, written below in the form provided in \cite{dixit-2013f}.
For $\alpha, \beta>0, \alpha\beta=1$, we have:
\begin{multline}\label{ferid}
\sqrt{\alpha} \int_{0}^{\infty} e^{ - \pi \alpha^{2} x^{2}} 
\left( \sum_{n=1}^{\infty} K_{0}(2 \pi n x) - \frac{1}{4x} \right) dx  \\
=\sqrt{\beta} \int_{0}^{\infty} e^{ - \pi \beta^{2} x^{2}} 
\left( \sum_{n=1}^{\infty} K_{0}(2 \pi n x) - \frac{1}{4x} \right) dx\\
= \frac{-1}{2 \pi^{3/2}} \int_{0}^{\infty} 
\Gamma \left( \frac{1+it}{4} \right)
\Gamma \left( \frac{1-it}{4} \right)
\Xi \left( \frac{t}{2} \right) 
\frac{\cos( \tfrac{1}{2} t \log \alpha) \, dt}{1+t^{2}}.
\end{multline}

\smallskip

The next theorem gives a generalization of Ferrar's result \eqref{ferid}.

\begin{theorem}\label{fergenthm}
Assume $-1 < \realpart{z} < 1$ and define 
\begin{multline}\label{frakf}
\mathfrak{F}(x,z) = x^{z/2} \Gamma \left( \frac{1+z}{2} \right) 
\left( \left[ 3 \gamma + 2 \log x + \psi \left( \tfrac{1}{2}(1+z) \right) 
\right] \zeta(1+z) + 2 \zeta'(1+z) \right) \\
\quad - \sqrt{\pi} \, \Gamma \left( \frac{z}{2} \right) \zeta(1-z)
x^{-z/2} + 
2 x^{z/2} \Gamma \left( \frac{1+z}{2} \right) 
\sum_{n=1}^{\infty} \sigma_{-z}(n) 
\left( \frac{n^{z}}{(x^{2}+n^{2})^{\tfrac{1+z}{2}} } - \frac{1}{n} 
\right).
\end{multline}
\noindent
Then, for $\alpha, \beta > 0$, $\alpha\beta=1$, 
\begin{multline}\label{fergen}
\sqrt{\alpha} \int_{0}^{\infty} K_{z/2}(2 \pi \alpha x) \mathfrak{F}(x,z) dx=\sqrt{\beta} 
\int_{0}^{\infty} K_{z/2}(2 \pi \beta x) \mathfrak{F}(x,z) dx\\
= \frac{8}{\pi} \int_{0}^{\infty} 
\Gamma \left( \frac{z+1+it}{4} \right) 
\Gamma \left( \frac{z+1-it}{4} \right) 
\Xi \left( \frac{t + iz}{2} \right)
\Xi \left( \frac{t - iz}{2} \right) \\
\quad \times \frac{\cos( \tfrac{1}{2} t \log \alpha ) dt}
{(t^{2} + (z+1)^{2})(t^{2} + (z-1)^{2})}.
\end{multline}
\end{theorem}

%

\begin{note}
The special value $z=0$ gives the identity
\begin{multline*}
\sqrt{\alpha} \int_{0}^{\infty} 
K_{0}(2 \pi \alpha x) \mathfrak{F}(x,0) dx =  \sqrt{\beta} \int_{0}^{\infty} 
K_{0}(2 \pi \beta x) \mathfrak{F}(x,0) dx\\
=\frac{8}{\pi} \int_{0}^{\infty} 
\Gamma \left( \frac{1+it}{4} \right) 
\Gamma \left( \frac{1-it}{4} \right) 
\frac{\Xi^{2} \left( t/2 \right)}{(1+t^{2})^{2}} 
\cos( \tfrac{1}{2} t \log \alpha ) \, dt,
\end{multline*}
\noindent
where the function $\mathfrak{F}(x,0)$ has the explicit form 
\begin{eqnarray}
\mathfrak{F}(x,0) & = & \frac{\pi^{2}}{24} + \frac{\gamma^{2}}{2} - 
\gamma \log 2 + \frac{\log^{2}2}{4} - \gamma_{1} + \gamma \log x \\
& + & \frac{\log x}{4} \log \left( \frac{x}{4} \right) + 
\frac{1}{2} \sum_{n=1}^{\infty} d(n) 
\left( \frac{1}{\sqrt{n^{2}+x^{2}}} - \frac{1}{n} \right),
\nonumber
\end{eqnarray}
\noindent
and where $\gamma, \, \gamma_{1}$ and $d(n)$ are as before. 
\end{note}

%
%

\section{Proofs}
\label{sec-proofs}

The Mellin transform
\begin{equation}
F(s) := \mathfrak{M}[f;s] = \int_{0}^{\infty} x^{s-1} f(x) dx 
\end{equation}
\noindent
is defined for a locally integrable function $f$. The existence of $F$ 
depends on the 
asymptotic behavior of $f$ at $x=0$ and $\infty$. In detail, if 
\begin{equation}
f(x) = \begin{cases}
O(x^{-a-\varepsilon}) & \quad \text{ as } x \to 0^{+}, \\
O(x^{-b+\varepsilon}) & \quad \text{ as } x \to + \infty,
\end{cases}
\end{equation}
\noindent
where $\varepsilon > 0$ and $a < b$, then $F(s)$ is an analytic function 
in the strip $a < \realpart{s} < b$. The properties of the Mellin transform 
used here include the inversion formula 
\begin{equation}
f(x) = \frac{1}{2 \pi i} \int_{c - i \infty}^{c + i \infty} 
x^{-s} F(s) ds
\end{equation}
\noindent
and Parseval's identity 
\begin{equation}
\int_{0}^{\infty} f(x) g(x) dx = \frac{1}{2 \pi i } 
\int_{c - i \infty}^{c + i \infty} F(1-s)G(s) ds,
\label{parseval-1}
\end{equation}
\noindent 
where the vertical line $\realpart{s} = c$ lies in the common strip of 
analyticity of the Mellin transforms $F(1-s)$ and $G(s)$. See \cite[p.~83]{paris-2001a} for the conditions
on the validity of this formula. A variant of this 
identity is
\begin{equation}
\int_{0}^{\infty} f(t) g \left( \frac{x}{t} \right) \frac{dt}{t} = 
\frac{1}{2 \pi i } \int_{c - i \infty}^{c + i \infty} 
x^{-s} F(s)G(s) ds.
\label{parseval-2}
\end{equation}
\noindent
For more details, see pages $79-83$ in \cite{paris-2001a}. 

\smallskip

The basic idea in the first method employs self-reciprocal functions. These 
are functions that are reproduced after integration against a kernel.
The key formulas required here are the remarkable identities 
derived by Koshlyakov \cite{koshlyakov-1938a} for $\tfrac{-1}{2}<\nu<\tfrac{1}{2}$:
\begin{equation}
\int_{0}^{\infty} K_{\nu}(t) \left( \cos(\nu \pi) M_{2 \nu}(2 \sqrt{xt}) -
\sin(\nu \pi) J_{2 \nu}(2 \sqrt{xt}) \right) dt = K_{\nu}(x),\label{koshlyakov-1}
\end{equation}
\noindent
and 
\begin{equation}
\int_{0}^{\infty} tK_{\nu}(t) \left( \sin(\nu \pi) J_{2 \nu}(2 \sqrt{xt}) -
\cos(\nu \pi) L_{2 \nu}(2 \sqrt{xt}) \right) dt = xK_{\nu}(x),\label{koshlyakov-2}
\end{equation}
\noindent 
where 
\begin{equation}
M_{\nu}(x) = \frac{2}{\pi} K_{\nu}(x) - Y_{\nu}(x) \text{ and }
L_{\nu}(x) = - \frac{2}{\pi} K_{\nu}(x) - Y_{\nu}(x)
\end{equation}
\noindent 
are the functions introduced by G.~H.~Hardy. It is easy to see that the above 
identities actually hold for  $ - \tfrac{-1}{2}<$ Re $\nu<\tfrac{1}{2}$.

\begin{lemma}
\label{lemma1}
Assume $ \pm \realpart{ \frac{z}{2}} < \realpart{s} < \tfrac{3}{4}$ and 
$y > 0$. Then 
\begin{multline}
\int_{0}^{\infty} x^{s-1} \left( \cos \left( \tfrac{1}{2}\pi z \right) 
M_{z}( 4 \pi \sqrt{xy}) - \sin \left( \tfrac{1}{2} \pi z \right) 
J_{z}(4 \pi \sqrt{xy}) \right) dx \\
= \frac{1}{2^{2s}\pi^{1+2s} y^{s}} 
\Gamma \left( s - \frac{z}{2} \right)
\Gamma \left( s + \frac{z}{2} \right)
\left( \cos \left( \tfrac{1}{2}\pi z \right) + \cos(\pi s) \right).
\end{multline}
\end{lemma}
\begin{proof}
The Mellin transform of 
the modified Bessel function $K_{z}(x)$, given in \eqref{mellin-KBessel}, is 
\begin{equation}
\int_{0}^{\infty} x^{s-1} K_{z}(ax)dx = 2^{s-2} a^{-s} 
\Gamma \left( \frac{s-z}{2} \right)
\Gamma \left( \frac{s+z}{2} \right)
\label{mellin-K}
\end{equation}
\noindent
for $\realpart{s} > \pm \realpart{z}$ and $a > 0$. 
This yields 
\begin{equation}
\int_{0}^{\infty} x^{s-1} \frac{2}{\pi} K_{z}(4 \pi \sqrt{xy}) dx = 
2^{-2s}\pi^{-1-2s}y^{-s} 
\Gamma \left( s - \frac{z}{2} \right)
\Gamma \left( s + \frac{z}{2} \right)
\end{equation}
for Re $s>\pm$ Re $\left(\frac{z}{2}\right)$ and $y>0$.

The Bessel function of the second kind $Y_{z}(x)$, satisfies 
\cite[p. 93, formula 
10.2]{oberhettinger-1974a} 
\begin{equation*}
\int_{0}^{\infty} x^{s-1} Y_{z}(ax)dx = - \frac{1}{\pi}2^{s-1} a^{-s} 
\cos \left( \tfrac{\pi}{2}(s-z) \right) 
\Gamma \left( \frac{s-z}{2} \right)
\Gamma \left( \frac{s+z}{2} \right),
\end{equation*}
\noindent
for $\pm$ Re $z<$ Re $s<\frac{3}{2}$, which produces 
\begin{equation*}
\int_{0}^{\infty} x^{s-1} Y_{z}(4 \pi \sqrt{xy})dx = 
- \pi^{-1-2s}2^{-2s} y^{-s} 
\cos \left( \pi s - \tfrac{1}{2} \pi z  \right) 
\Gamma \left( s - \frac{z}{2} \right)
\Gamma \left( s + \frac{z}{2} \right),
\end{equation*}
\noindent 
for $\pm \realpart{\frac{z}{2}} < \realpart{s} < \tfrac{3}{4}$. Thus for $\pm \realpart{ \frac{z}{2}} < \realpart{s} < \tfrac{3}{4}$,
\begin{multline*}
\int_{0}^{\infty} x^{s-1} \cos \left( \tfrac{1}{2} \pi z \right) 
M_{z}( 4 \pi \sqrt{xy}) dx \\
= 2^{-2s} \pi^{-1-2s} y^{-s} 
\Gamma \left( s - \frac{z}{2} \right)
\Gamma \left( s + \frac{z}{2} \right)
\cos \left( \tfrac{1}{2} \pi z \right) 
\left( 1 + \cos \left( \pi s - \tfrac{1}{2} \pi z \right) \right).
\end{multline*}

Similarly, using 
\cite[p. 93, formula 10.1]{oberhettinger-1974a}, 
\begin{equation}
\label{Mellin-besselJ}
\int_{0}^{\infty} x^{s-1} J_{z}(ax) \, dx = 
\frac{1}{2} \left( \frac{a}{2} \right)^{-s} \frac{\Gamma \left( \frac{s+z}{2}
\right)}{\Gamma \left( 1 + \frac{z-s}{2} \right)}
\end{equation}
for $- \realpart{z}< \realpart{s}<\frac{3}{2}$, yields 
\begin{equation*}
\int_{0}^{\infty} x^{s-1} J_{z}( 4 \pi \sqrt{xy} ) dx = 
2^{-2s} \pi^{-2s} y^{-s} 
\frac{\Gamma \left( s + \tfrac{z}{2} \right)}
{\Gamma \left( 1 - s + \tfrac{z}{2} \right)}
\end{equation*}
for $-\realpart{\frac{z}{2}} < \realpart{s}<\frac{3}{4}$. Combining
these evaluations completes the proof.
\end{proof}

Similar arguments prove the next result. 

\begin{lemma}
\label{lemma2}
Assume $ \pm \realpart{ \frac{z}{2}} < \realpart{s} < \tfrac{3}{4}$ and $y > 0$. Then 
\begin{multline}
\int_{0}^{\infty} x^{s-1} \left( \sin \left( \tfrac{1}{2}\pi z \right) 
J_{z}( 4 \pi \sqrt{xy}) - \cos \left( \tfrac{1}{2} \pi z \right) 
L_{z}(4 \pi \sqrt{xy}) \right) dx \\
= \frac{1}{2^{2s}\pi^{1+2s} y^{s}} 
\Gamma \left( s - \frac{z}{2} \right)
\Gamma \left( s + \frac{z}{2} \right)
\left( \cos \left( \tfrac{1}{2}\pi z \right) - \cos(\pi s) \right).
\end{multline}
\end{lemma}

The next statement shows how to produce functions self-reciprocal with 
respect to the kernel \eqref{kernel-100}.

\begin{theorem}
\label{production-1}
Assume $\pm \realpart{\frac{z}{2}} < c = \realpart{s} < 1 \pm 
\realpart{\frac{z}{2}}$. Define $f(x, z)$ by 
\begin{equation}
f(x,z) = \frac{1}{2 \pi i} \int_{c - i \infty}^{c + i \infty} 
x^{-s} F(s,z) \zeta(1 - s -z/2) \zeta(1-s + z/2) \, ds,
\end{equation}
where $F(s,z)$ is a function satisfying $F(s,z) = F(1-s,z)$ and is such that the above integral converges.
\noindent
Then $f$ is self-reciprocal (as a function of $x$) with respect to the kernel 
\begin{equation}
\label{kernel-100}
2 \pi \left( \cos \left( \tfrac{1}{2} \pi z \right) M_{z}( 4 \pi \sqrt{xy}) -
\sin \left( \tfrac{1}{2} \pi z \right) J_{z}( 4 \pi \sqrt{xy}) \right),
\end{equation}
\noindent
that is,
\begin{equation}\label{thm5.3fyz}
f(y,z) = 2 \pi \int_{0}^{\infty} f(x,z) 
\left[ \cos \left( \tfrac{1}{2} \pi z \right) M_{z}( 4 \pi \sqrt{xy}) -
\sin \left( \tfrac{1}{2} \pi z \right) J_{z}( 4 \pi \sqrt{xy}) \right] dx.
\end{equation}
\end{theorem}
\begin{proof}
Apply Parseval's identity \eqref{parseval-1} to the product of 
the functions $f(x,z)$ and $g(x,z) = 
2 \pi \left[ \cos \left( \tfrac{1}{2} \pi z \right) M_{z}( 4 \pi \sqrt{xy}) -
\sin \left( \tfrac{1}{2} \pi z \right) J_{z}( 4 \pi \sqrt{xy}) \right]$, 
with $z$ a parameter. Lemma \ref{lemma1} 
gives the identity 
\begin{multline*}
2 \pi \int_{0}^{\infty} f(x,z) \left[ \cos \left( \tfrac{1}{2} \pi z \right)
M_{z}(4 \pi \sqrt{xy}) - \sin \left( \tfrac{1}{2} \pi z \right) 
J_{z}(4 \pi \sqrt{xy}) \right] dx = \\
\frac{1}{2 \pi i} \int_{C} F(s,z) 
\zeta \left( s - \frac{z}{2} \right)
\zeta \left( s + \frac{z}{2} \right)
2^{1-2s} \pi^{-2s} y^{-s}  \\
\times \Gamma \left( s - \frac{z}{2} \right)
\Gamma \left( s + \frac{z}{2} \right)
\left[ \cos \left( \tfrac{1}{2} \pi z \right) + \cos(\pi s ) \right] ds,
\end{multline*}
\noindent
where $C$ is the line $\realpart{s} = c$ and $\pm$ Re $\frac{z}{2}<c< \text{min}\left\{\frac{3}{4}, 1\pm\text{Re}\, \frac{z}{2}\right\}$. Now use the 
functional equation for 
the Riemann zeta function \eqref{functional-zeta} in the form 
\begin{equation}
\zeta(s) = 2^{s} \pi^{s-1} \Gamma(1-s) \zeta(1-s) 
\sin \left( \tfrac{1}{2} \pi s \right)\label{zetafe}
\end{equation}
\noindent
to simplify the line integral and produce 
\begin{multline*}
2 \pi \int_{0}^{\infty} f(x,z) \left[ \cos \left( \tfrac{1}{2} \pi z \right)
M_{z}(4 \pi \sqrt{xy}) - \sin \left( \tfrac{1}{2} \pi z \right) 
J_{z}(4 \pi \sqrt{xy}) \right] dx = \\
\frac{1}{2 \pi i } \int_{C} F(s,z) 
\zeta \left(1-s- \frac{z}{2} \right)
\zeta \left(1-s+ \frac{z}{2} \right)y^{-s} ds.
\end{multline*}
\noindent
This last line is $f(y,z)$ and the result has been established.
\end{proof}

\begin{corollary}\label{cor5.4}
Let $f(x,z)$ be as in the previous theorem. Then, if $\alpha, \, \beta > 0$ 
and $\alpha \beta = 1$ and $-1 < \realpart{z} < 1$, 
\begin{equation}\label{cor5.4result}
\sqrt{\alpha} \int_{0}^{\infty} K_{z/2}(2 \pi \alpha x) f(x,z) dx = 
\sqrt{\beta} \int_{0}^{\infty} K_{z/2}(2 \pi \beta x) f(x,z) dx.
\end{equation}
\end{corollary}
\begin{proof}
The identity \eqref{koshlyakov-1} produces 
\begin{equation}\label{selfkoshcv1}
K_{\frac{z}{2}}(2\pi\alpha x)=\frac{2\pi}{\alpha}\int_{0}^{\infty}K_{\frac{z}{2}}\left(\frac{2\pi y}{\alpha}\right)\left(\cos\left(\frac{\pi z}{2}\right)M_{z}(4\pi\sqrt{yx})-\sin\left(\frac{\pi z}{2}\right)J_{2\nu}(4\pi\sqrt{yx})\right)\, dy.
\end{equation}
Thus,
\begin{equation*}
\int_{0}^{\infty}K_{\frac{z}{2}}(2\pi\alpha x)f(x,z)\, dx  =  
\end{equation*}
\begin{equation*}
\frac{2\pi}{\alpha}\int_{0}^{\infty}K_{\frac{z}{2}}
\left(\frac{2\pi y}{\alpha}\right)\int_{0}^{\infty}f(x,z) 
\left(\cos\left(\frac{\pi z}{2}\right)M_{z}(4\pi\sqrt{yx})-
\sin\left(\frac{\pi z}{2}\right)J_{2\nu}(4\pi\sqrt{yx})\right)\, dx\, dy,
\end{equation*}
where the interchange of the order of integration can be easily justified. Now 
apply \eqref{thm5.3fyz} and use the fact $\alpha\beta=1$, to 
deduce \eqref{cor5.4result}.
\end{proof}

The next statement admits a similar proof as the previous theorem.

\begin{theorem}
\label{production-2}
Assume $\pm \realpart{\frac{z}{2}} < c = \realpart{s} < 1 \pm 
\realpart{\frac{z}{2}}$. Define $f(x, z)$ by
\begin{equation}
f(x,z) = \frac{1}{2 \pi i} \int_{c - i \infty}^{c + i \infty} 
\frac{F(s,z)}{(2 \pi)^{2s}}  
\Gamma \left( s - \frac{z}{2} \right) 
\Gamma \left( s + \frac{z}{2} \right) 
\zeta\left(s -\frac{z}{2}\right) \zeta\left(s + \frac{z}{2}\right) x^{-s} \, ds,
\end{equation}
where $F(s,z)$ is a function satisfying $F(s,z) = F(1-s,z)$ and is such that the above integral converges.
\noindent
Then $f$ is self-reciprocal (as a function of $x$) with respect to the kernel 
\begin{equation}
2 \pi \left( \sin \left( \tfrac{1}{2} \pi z \right) J_{z}( 4 \pi \sqrt{xy}) -
\cos \left( \tfrac{1}{2} \pi z \right) L_{z}( 4 \pi \sqrt{xy}) \right),
\end{equation}
\noindent
that is,
\begin{equation}\label{thm5.5fyz}
f(y,z) = 2 \pi \int_{0}^{\infty} f(x,z) 
\left[ \sin \left( \tfrac{1}{2} \pi z \right) J_{z}( 4 \pi \sqrt{xy}) -
\cos \left( \tfrac{1}{2} \pi z \right) L_{z}( 4 \pi \sqrt{xy}) \right] dx.
\end{equation}
\end{theorem}

\begin{corollary}\label{cor5.6}
Let $f(x,z)$ be as in the previous theorem. Then, if $\alpha, \, \beta > 0$ 
and $\alpha \beta = 1$ and $-1 < \realpart{z} < 1$, 
\begin{equation}
\sqrt{\alpha^{3}} \int_{0}^{\infty} x K_{z/2}(2 \pi \alpha x) f(x,z) dx = 
\sqrt{\beta^{3}} \int_{0}^{\infty} x K_{z/2}(2 \pi \beta x) f(x,z) dx.
\end{equation}
\end{corollary}
\begin{proof}
The proof of similar to that of Corollary \ref{cor5.4}, so details are 
omitted.
\end{proof}
\smallskip

The proofs of the theorems stated in
Section \ref{sec-param-intro} are given now. In the proofs below, the 
convergence of the integrals involving the gamma function
can be easily proved using Stirling's formula  for 
$\Gamma(s)$, $s=\sigma+it$, in a vertical 
strip $\alpha\leq\sigma\leq\beta$ given by
\begin{equation}\label{strivert}
|\Gamma(\sigma+it)|=(2\pi)^{\tfrac{1}{2}}|t|^{\sigma-\tfrac{1}{2}}e^{-\tfrac{1}{2}\pi |t|}\left(1+O\left(\frac{1}{|t|}\right)\right)
\end{equation}
as $|t|\to\infty$. The 
vanishing of the integrals which involve gamma function along 
the horizontal segments of a contour as the height $T\to\infty$ is 
also established using \eqref{strivert}.

The interchange of the order of summation and integration, or of the order of integration
in the case of double integrals, is permissible because of absolute convergence of the series and integrals involved.

\smallskip

\noindent
\textit{\textbf{Proof of Theorem \ref{thm-46}}}. Take 
$F(s,z) \equiv 1$ in Theorem \ref{production-2}. The conclusion requires 
the evaluation of 
\begin{multline}
\label{form-f}
f(x,z) = \frac{1}{2 \pi i } \int_{c - i \infty}^{c + i \infty} \left(4\pi^2 x\right)^{-s}\\
\times \Gamma \left( s - \frac{z}{2} \right)
\Gamma \left( s + \frac{z}{2} \right)
\zeta \left( s - \frac{z}{2} \right)
\zeta \left( s + \frac{z}{2} \right)  ds
\end{multline}
\noindent
for $ \pm \realpart{ \tfrac{z}{2} } < c = \realpart{s} < 1 \pm 
\realpart{\tfrac{z}{2}}$. It will be shown next that 
\begin{equation}
f(x,z) = \Phi(x,z), 
\end{equation}
\noindent
using the notation of Theorem \ref{thm-46}. 

\medskip

This computation begins by using the expansion 
\begin{equation}
\zeta \left( s - \frac{z}{2} \right) 
\zeta \left( s + \frac{z}{2} \right)  = \sum_{n=1}^{\infty} 
\frac{\sigma_{-z}(n)}{n^{s - z/2}}
\label{expansion-sigma}
\end{equation}
\noindent
valid for $\realpart{s} > 1  \pm \realpart{ \frac{z}{2}}$. See 
\cite[p. $8$, equation $1.3.1$]{titchmarsh-1986a}. Thus, in order to use 
\eqref{expansion-sigma} in \eqref{form-f}, it is required to shift the 
line of integration to the vertical line $\lambda = \realpart{s} > 
1 \pm \realpart{\frac{z}{2}}$. This shift captures
two poles of the integrand at $s = 1 +z/2$ and $s = 1 - z/2$. Since the 
contributions of the horizontal segments at $\imagpart{s} = \pm T$ vanish
as $T \to \infty$, the residue theorem gives 
\begin{multline*}
f(x,z)  =  \sum_{n=1}^{\infty} \sigma_{-z}(n) n^{z/2} 
\frac{1}{2 \pi i} \int_{\lambda - i \infty}^{\lambda + i \infty} 
\Gamma \left( s - \frac{z}{2} \right) 
\Gamma \left( s + \frac{z}{2} \right) 
(4 \pi^{2} n x)^{-s} ds \\
 -   \lim\limits_{s \to 1 + z/2} 
\left( s - \frac{z}{2} - 1 \right) 
\zeta \left( s - \frac{z}{2} \right)
\zeta \left( s + \frac{z}{2} \right)
\Gamma \left( s - \frac{z}{2} \right)
\Gamma \left( s + \frac{z}{2} \right)
(4 \pi^{2} x)^{-s} \\
\quad -  \lim\limits_{s \to 1 - z/2} 
\left( s + \frac{z}{2} - 1 \right) 
\zeta \left( s - \frac{z}{2} \right)
\zeta \left( s + \frac{z}{2} \right)
\Gamma \left( s - \frac{z}{2} \right)
\Gamma \left( s + \frac{z}{2} \right)
(4 \pi^{2} x)^{-s}.
\end{multline*}

The line integral above is evaluated using \eqref{mellin-K} that gives 
\begin{equation}
\frac{1}{2 \pi i} \int_{\lambda - i \infty}^{\lambda + i \infty} \Gamma \left( s - 
\frac{z}{2} \right) \Gamma\left( s + \frac{z}{2} \right) 
(4 \pi^{2} nx)^{-s} ds = 2 K_{z}(4 \pi \sqrt{nx}). 
\end{equation}
\noindent
Then the residues at the poles are computed using 
\begin{equation}
\lim\limits_{s \to 1} (s-1) \zeta(s) = 1,
\end{equation}
\noindent 
that gives 
\begin{multline*}
f(x,z) = 2 \sum_{n=1}^{\infty} \sigma_{-z}(n) n^{z/2} K_{z}(4 \pi \sqrt{nx})
\\
- \frac{\Gamma(1+z) \zeta(1+z) x^{-1-z/2}}{(2 \pi)^{2 + z}} 
- \frac{\Gamma(1-z) \zeta(1-z) x^{-1+z/2}}{(2 \pi)^{2 - z}}.
\end{multline*}
\noindent
This shows $f(x,z) = \Phi(x,z)$. Corollary \eqref{cor5.6} now 
establishes the first equality in \eqref{thm-46id}.

\smallskip

The equality between the extreme left and right sides of \eqref{thm-46id} is 
establihsed next. The invariance of the latter under $\alpha\to 1/\alpha$ then 
easily establishes the other equality, thus giving another proof of the 
transformation.

The proof begins by replacing $s$ by $s+1$, $a$ by $2 \pi \alpha$ and 
$z$ by $z/2$ in 
\eqref{mellin-K} to obtain
\begin{multline}
\label{mell-1}
\frac{1}{2 \pi i} \int_{c - i \infty}^{c + i \infty} 2^{s-1} 
(2 \pi \alpha)^{-s-1} 
\Gamma \left( \frac{s+1}{2} - \frac{z}{4} \right)
\Gamma \left( \frac{s+1}{2} + \frac{z}{4} \right)
x^{-s} ds  \\
= x K_{z/2}( 2 \pi \alpha x)
\end{multline}
for Re $s>-1\pm$ Re $\frac{z}{2}$.
\noindent
Then \eqref{parseval-1}, \eqref{form-f} and \eqref{mell-1} give
\begin{multline}
\sqrt{\alpha^{3}} \int_{0}^{\infty} x K_{z/2}( 2 \pi \alpha x) \Phi(x,z) dx = \\
\frac{\sqrt{\alpha^{3}}}{2 \pi i} 
\int_{c - i \infty}^{c + i \infty} 
\left[ \frac{ (\pi \alpha)^{s-2}}{4} 
\Gamma \left( 1 - \frac{s}{2} - \frac{z}{4} \right)
\Gamma \left( 1 - \frac{s}{2} + \frac{z}{4} \right)
\right] \\
\times \frac{1}{( 2 \pi)^{2s}} 
\Gamma \left( s - \frac{z}{2} \right)
\Gamma \left( s + \frac{z}{2} \right)
\zeta \left( s - \frac{z}{2} \right)
\zeta \left( s + \frac{z}{2} \right)
\, ds,
\end{multline}
\noindent
for $\pm \realpart{ \frac{z}{2}} < c = \realpart{s} < 1 \pm 
\realpart{ \frac{z}{2} }$. Now use the relation
\begin{equation}
\Gamma(s) \Gamma \left( s + \tfrac{1}{2} \right) = 
\frac{\sqrt{\pi}}{2^{2s-1}} \Gamma(2s)
\end{equation}
\noindent
to express each of $\Gamma(s \pm z/2)$ as a product of two gamma factors and 
obtain
\begin{multline}
\sqrt{\alpha^{3}} \int_{0}^{\infty} x K_{z/2}( 2 \pi \alpha x ) \Phi(x,z) dx = 
\\
\frac{1}{32 \pi^{4} i \sqrt{\alpha}} 
\int_{c - i \infty}^{c + i \infty} 
\Gamma \left( 1 - \frac{s}{2} - \frac{z}{4} \right) 
\Gamma \left( 1 - \frac{s}{2} + \frac{z}{4} \right) 
\Gamma \left( \frac{s}{2} - \frac{z}{4} \right) 
\Gamma \left( \frac{s}{2} + \frac{z}{4} \right)  \\
\times \Gamma \left( \frac{s}{2} - \frac{z}{4} + \frac{1}{2} \right) 
\Gamma \left( \frac{s}{2} + \frac{z}{4} + \frac{1}{2} \right) 
\zeta \left( s - \frac{z}{2} \right) 
\zeta \left( s + \frac{z}{2} \right) 
\left( \frac{\pi}{\alpha} \right)^{-s} ds.
\end{multline}
The integrand is now expressed 
in terms of the Riemann $\xi$-function \eqref{def-xi}
with $c = \tfrac{1}{2}$ and $-1 < \realpart{z} < 1$. This yields 
\begin{multline}\label{final4.6}
\sqrt{\alpha^{3}} \int_{0}^{\infty} x K_{z/2}( 2 \pi \alpha x ) \Phi(x,z) dx = 
\\
\frac{1}{128 \pi^{4} i \sqrt{\alpha}} 
\int_{\tfrac{1}{2} - i \infty}^{\tfrac{1}{2} + i \infty} 
\Gamma \left( - \frac{s}{2} - \frac{z}{4} \right) 
\Gamma \left( - \frac{s}{2} + \frac{z}{4} \right) 
\Gamma \left( \frac{s}{2} - \frac{z}{4} - \frac{1}{2} \right)  \\
\times \Gamma \left( \frac{s}{2} + \frac{z}{4} - \frac{1}{2} \right)  
\xi \left( s - \frac{z}{2} \right)
\xi \left( s + \frac{z}{2} \right)\alpha^{s} ds.
\end{multline}

The last step is to use the identity \footnote{There is a typo in the identity in \cite{dixit-2011b}.
One of the $\Xi$-functions should have $(t-iz)/2$ as its argument.}
\begin{multline}
\label{formula-210}
\int_{0}^{\infty} f \left(z, \frac{t}{2} \right)
\Xi \left( \frac{t + i z}{2} \right)
\Xi \left( \frac{t - i z}{2} \right)
\cos \left( \tfrac{1}{2} t \log \alpha \right) dt = \\
\frac{1}{i \sqrt{\alpha}} 
\int_{\tfrac{1}{2} - i \infty}^{\tfrac{1}{2} + i \infty} 
\phi \left( z, s - \tfrac{1}{2} \right)
\phi \left( z, \tfrac{1}{2} -s \right)
\xi \left( s -  \frac{z}{2} \right)
\xi \left( s +  \frac{z}{2} \right)
\alpha^{s} ds,
\end{multline}
\noindent
established in \cite{dixit-2011b}, with 
$f(z,t) = \phi(z,it) \phi(z,-it)$, to rewrite the integral on the right side of \eqref{final4.6}. Taking 
\begin{equation}
\phi(z,s) = \frac{1}{8 \sqrt{2} \, \pi^{2}} 
\Gamma \left( - \frac{s}{2} + \frac{z}{4} - \frac{1}{4} \right)
\Gamma \left( - \frac{s}{2} - \frac{z}{4} - \frac{1}{4} \right)
\end{equation}
\noindent
produces 
\begin{multline}
\sqrt{\alpha^{3}} \int_{0}^{\infty} x K_{z/2}( 2 \pi \alpha x) \Phi(x,z) dx
= \\
\frac{1}{128 \pi^{4}} 
\int_{0}^{\infty} 
\Gamma \left( \frac{z-1 + it}{4} \right) 
\Gamma \left( \frac{z-1 - it}{4} \right) 
\Gamma \left( \frac{-z-1 + it}{4} \right) \\
\times 
\Gamma \left( \frac{-z-1 - it}{4} \right) 
\Xi \left( \frac{t+iz}{2} \right)
\Xi \left( \frac{t-iz}{2} \right)
\cos \left( \tfrac{1}{2} t \log \alpha \right) dt.
\end{multline}
\noindent
Finally, the integrand on the right side can be simplified, using 
$\Gamma(u+1) = u \Gamma(u)$, to the form given 
in \eqref{thm-46id}. This completes the proof.

\medskip

\textit{\textbf{Proof of Theorem \ref{jbesstra}}}. In view of Corollary \ref{cor-47}, it suffices to show that
\begin{align}\label{ramdit}
&\\
&\sqrt{\alpha^{3}} 
\int_{0}^{\infty} x K_{0}(2 \pi \alpha x) 
\left( 2 \sum_{n=1}^{\infty} d(n) K_{0}(4 \pi \sqrt{nx}) + 
\frac{\log( 4 \pi^{2} x)}{4 \pi^{2}x} \right) dx\nonumber\\
&=\frac{\sqrt{\alpha}}{16\pi^2}\int_{0}^{\infty}\int_{0}^{\infty}\frac{y}{(y^2+t^2)^{3/2}}\left(J_{0}(2\alpha y)+\frac{4\pi t}{e^{2\pi t}-1}\left(\frac{1}{e^{2\pi\alpha y}-1}-\frac{1}{2\pi\alpha y}\right)\right)\, dy\, dt.\nonumber
\end{align}

The proof begins with an auxiliary result.

\begin{lemma}\label{lemdit}
For $\alpha, t>0$, we have
\begin{align}
\label{lambertkbess}
\int_{0}^{\infty}\frac{xK_{0}(2\pi\alpha x)}{e^{2\pi x/t}-1}\, dx=\frac{t}{8\pi\alpha}-\frac{t^2}{2}\int_{0}^{\infty}xJ_{0}(2\pi t\alpha x)\left(\psi(x+1)-\log x\right)\, dx.
\end{align}
\end{lemma}
\begin{proof}
Let $z=0$ in \eqref{mell-1} to see that for $\lambda=$ Re $s>-1$,
\begin{equation}\label{bessinv}
xK_{0}(2\pi\alpha x)=\frac{1}{2 \pi i} \int_{\lambda - i \infty}^{\lambda + i \infty} 
2^{s-1} (2\pi\alpha)^{-s-1} \Gamma^{2}\left( \frac{s+1}{2}\right) x^{-s} \, ds. 
\end{equation}
It is well-known that for Re $s>1$,
\begin{equation}
\int_{0}^{\infty}\frac{y^{s-1}}{e^{y}-1}\, dy=\Gamma(s)\zeta(s).
\end{equation}
Let $y=2\pi x/t$ with $t>0$ to obtain 
\begin{equation}\label{gammazeta}
\frac{1}{2 \pi i} \int_{\lambda' - i \infty}^{\lambda' + i \infty} \Gamma(s)\zeta(s)\left(\frac{2\pi x}{t}\right)^{-s}\, ds=\frac{1}{e^{2\pi x/t}-1},
\end{equation}
\noindent
for $\lambda' = \realpart{s} > 1$.  Using 
(\ref{bessinv}), (\ref{gammazeta}) and Parseval's identity 
\eqref{parseval-1}, for $-1<c=$ Re $s<0$, give 
\begin{align*}
&\int_{0}^{\infty}\frac{xK_{0}(2\pi\alpha x)}{e^{2\pi x/t}-1}\, dx\nonumber\\
&=\frac{1}{2 \pi i} \int_{c - i \infty}^{c + i \infty} \left(\frac{t}{2\pi}\right)^{1-s}\Gamma(1-s)\zeta(1-s)2^{s-1} (2\pi\alpha)^{-s-1} \Gamma^{2}\left( \frac{s+1}{2}\right)\, ds.\nonumber
\end{align*}
Now use the functional equation \eqref{zetafe}
and a variant of the reflection formula for the gamma function, namely,
\begin{equation}
\Gamma\left(\frac{1}{2}+w\right)\Gamma\left(\frac{1}{2}-w\right)=\frac{\pi}{\cos\pi w}, \hspace{3mm}w-\tfrac{1}{2}\notin\mathbb{Z},
\end{equation}
in the above equation and simplify to deduce that
\begin{equation}\label{lambertbess1}
\int_{0}^{\infty}\frac{xK_{0}(2\pi\alpha x)}{e^{2\pi x/t}-1}\, dx=\frac{t}{8\pi i\alpha}\int_{c - i \infty}^{c + i \infty}\frac{\Gamma\left(\frac{1+s}{2}\right)}{\Gamma\left(\frac{1-s}{2}\right)}\frac{\zeta(s)}{\sin\pi s}(t\pi\alpha)^{-s}\, ds.
\end{equation}
Next, shift the line of integration from Re $s=c$, $-1<c<0$, to 
Re $s=c'$, $0<c'<1/2$ and apply the residue theorem by considering a 
rectangular contour. In doing so, one needs  to consider the contribution 
from the pole of order $1$ at $s=0$. Noting that the integrals along the horizontal segments of the contour tend to zero as the height tends to $\infty$, 
gives
\begin{align}\label{lambertbess2} \\
&\int_{0}^{\infty}\frac{xK_{0}(2\pi\alpha x)}{e^{2\pi x/t}-1}\, dx\nonumber\\
&=\frac{t}{4\alpha}\left\{\frac{1}{2 \pi i}\int_{c' - i \infty}^{c' + i \infty}\frac{\Gamma\left(\frac{1+s}{2}\right)}{\Gamma\left(\frac{1-s}{2}\right)}\frac{\zeta(s)}{\sin\pi s}(t\pi\alpha)^{-s}\, ds-\lim_{s\to 0}s\frac{\Gamma\left(\frac{1+s}{2}\right)}{\Gamma\left(\frac{1-s}{2}\right)}\frac{\zeta(s)}{\sin\pi s}(t\pi\alpha)^{-s}\right\}\nonumber\\
&=\frac{t}{4\alpha}\left(\frac{1}{2 \pi i}\int_{c' - i \infty}^{c' + i \infty}\frac{\Gamma\left(\frac{1+s}{2}\right)}{\Gamma\left(\frac{1-s}{2}\right)}\frac{\zeta(s)}{\sin\pi s}(t\pi\alpha)^{-s}\, ds+\frac{1}{2\pi}\right).
\nonumber
\end{align}
To evaluate the above integral, first use \eqref{Mellin-besselJ} with $z=0, a=2$ and $s$ replaced by $s+1$ so that for $-1<d=$ Re $s<\tfrac{1}{2}$,
\begin{equation}\label{jbess}
\frac{1}{2\pi i}\int_{d-i\infty}^{d+i\infty}\frac{\Gamma\left(\frac{1+s}{2}\right)}{\Gamma\left(\frac{1-s}{2}\right)}x^{-s}\, ds=2xJ_{0}(2x).
\end{equation}
The next steps employs a formula of Kloosterman 
\cite[p. 24-25, equations (2.9.1), (2.9.2)]{titchmarsh-1986a}, for 
$0<d'=$ Re $s<1$,
\begin{equation}
\frac{\zeta(s)}{\sin\pi s}=-\frac{1}{\pi}\int_{0}^{\infty}\left(\psi(x+1)-\log x\right)x^{-s}\, dx.
\end{equation}
Replacing $x$ by $1/x$ in the above formula, produces 
\begin{equation}\label{kloos}
\frac{1}{2\pi i}\int_{d'-i\infty}^{d'+i\infty}\frac{\zeta(s)}{\sin\pi s}x^{-s}\, ds=-\frac{1}{\pi x}\left(\psi\left(\frac{1}{x}+1\right)+\log x\right),
\end{equation}
for $0<d'=$ Re $s<1$. Since 
(\ref{jbess}) and (\ref{kloos}) are both valid in the region $0<$ Re $s<\tfrac{1}{2}$, using \eqref{parseval-2}, it follows that
\begin{align}
&\frac{1}{2 \pi i}\int_{c' - i \infty}^{c' + i \infty}\frac{\Gamma\left(\frac{1+s}{2}\right)}{\Gamma\left(\frac{1-s}{2}\right)}\frac{\zeta(s)}{\sin\pi s}(t\pi\alpha)^{-s}\, ds\nonumber\\
&=\frac{-1}{\pi^2t\alpha}\int_{0}^{\infty}2xJ_{0}(2x)\left(\psi\left(\frac{x}{\pi t\alpha}+1\right)-\log\frac{x}{\pi t\alpha}\right)\, dx.
\end{align}
Now let $x\to\pi t\alpha x$ and substitute in (\ref{lambertbess2}) to 
obtain (\ref{lambertkbess}).
\end{proof}

A proof of \eqref{ramdit} is presented next. 

\begin{proof}
Let
\begin{equation}\label{halpha}
H(\alpha):=\sqrt{\alpha^{3}} 
\int_{0}^{\infty} x K_{0}(2 \pi \alpha x) 
\left( 2 \sum_{n=1}^{\infty} d(n) K_{0}(4 \pi \sqrt{nx}) + 
\frac{\log( 4 \pi^{2} x)}{4 \pi^{2}x} \right) dx.
\end{equation}
Page 254 in the Lost Notebook \cite{ramanujan-1988a} (see also equation 
(4.1) in \cite{berndt-2008a}) gives 
\begin{equation}
\int_{0}^{\infty}\frac{dt}{t(e^{2\pi t}-1)(e^{2\pi x/t}-1)}=2\sum_{n=1}^{\infty}d(n)K_{0}\left(4\pi\sqrt{nx}\right).
\end{equation}
Hence 
\begin{align}\label{integ1}
H(\alpha)&=\sqrt{\alpha^{3}}\bigg\{\frac{ \log( 4 \pi^{2})}{4\pi^2}\int_{0}^{\infty}K_{0}(2\pi\alpha x)\, dx+\frac{1}{4\pi^2}\int_{0}^{\infty}K_{0}(2\pi\alpha x)\log x\, dx\nonumber\\
&\quad\quad\quad\quad+\int_{0}^{\infty}xK_{0}(2\pi\alpha x)\, dx\int_{0}^{\infty}\frac{dt}{t(e^{2\pi t}-1)(e^{2\pi x/t}-1)}\bigg\}.
\end{align}
Now formula $6.511.12$ in \cite{gradshteyn-2007a} gives 
\begin{equation}
\int_{0}^{\infty}K_{0}(2\pi\alpha x)\, dx=\frac{1}{4\alpha},\label{int-K0}
\end{equation}
and formula $2.16.20.1$ on page $365$ of \cite{prudnikov-1986a} states 
that for 
$| \realpart{w} | > \realpart{\nu}$ and real $m>0$,
\begin{multline}\label{mkl}
\int_{0}^{\infty} x^{w-1} K_{\nu}(mx) \log x \, dx = 
\frac{2^{w-3}}{m^{w}} 
\Gamma \left( \frac{w + \nu}{2} \right)
\Gamma \left( \frac{w - \nu}{2} \right) \\
\quad \left\{ \psi \left( \frac{w + \nu}{2} \right) +
\psi \left( \frac{w - \nu}{2} \right)
-2 \log \left( \frac{m}{2} \right) \right\}. 
\end{multline}
\noindent
Then \eqref{int-K0} and the above formula with $w=1$, $\nu=0$ 
and $m=2\pi\alpha$ converts  (\ref{integ1}) to 
\begin{align}\label{integ2}
H(\alpha)&=\sqrt{\alpha^{3}}\bigg\{\frac{ \log( 4 \pi^{2})}{16\pi^2\alpha}-\frac{1}{4\pi^2}\left(\frac{\gamma+\log(4\pi\alpha)}{4\alpha}\right)
+\int_{0}^{\infty}\frac{dt}{t(e^{2\pi t}-1)}\int_{0}^{\infty}\frac{xK_{0}(2\pi\alpha x)}{e^{2\pi x/t}-1}\, dx\bigg\}\nonumber\\
&=\frac{\sqrt{\alpha}}{16\pi^2}\left\{-\left(\gamma-\log\left(\frac{\pi}{\alpha}\right)\right)+16\pi^2\alpha\int_{0}^{\infty}\frac{dt}{t(e^{2\pi t}-1)}\int_{0}^{\infty}\frac{xK_{0}(2\pi\alpha x)}{e^{2\pi x/t}-1}\, dx\right\}.
\end{align} 
Lemma \ref{lemdit} along with the 
integral representation \cite[equation 3.5]{berndt-2010a}
\begin{equation}\label{berndtdixit}
\gamma-\log\left(\frac{\pi}{\alpha}\right)=\int_{0}^{\infty}\left(\frac{2\pi}{e^{2\pi t}-1}-\frac{e^{-2\alpha t}}{t}\right)\, dt
\end{equation}
is used in (\ref{integ2}) to obtain
\begin{align}\label{integ3}
H(\alpha)&=\frac{\sqrt{\alpha}}{16\pi^2}\bigg\{\int_{0}^{\infty}\left(\frac{e^{-2\alpha t}}{t}-\frac{2\pi}{e^{2\pi t}-1}\right)\, dt\nonumber\\
&\quad\quad\quad\quad+\int_{0}^{\infty}\frac{dt}{t(e^{2\pi t}-1)}\left(2\pi t-8\pi^2\alpha t^2\int_{0}^{\infty}xJ_{0}(2\pi t\alpha x)\left(\psi(x+1)-\log x\right)\, dx\right)\bigg\}\nonumber\\
&=\frac{\sqrt{\alpha}}{16\pi^2}\int_{0}^{\infty}\left(\frac{e^{-2\alpha t}}{t}-\frac{8\pi^2\alpha t}{e^{2\pi t}-1}\int_{0}^{\infty}xJ_{0}(2\pi t\alpha x)\left(\psi(x+1)-\log x\right)\, dx\right)\, dt.
\end{align}
The standard formulas
\begin{equation}\label{psilog}
2\pi\int_{0}^{\infty}\left(\frac{1}{e^{2\pi y}-1}-\frac{1}{2\pi y}\right)e^{-2\pi xy}\, dy=\log x-\psi(x+1),
\end{equation}
which can be obtained from 
\cite[p.~360, 3.427.7]{gradshteyn-2007a} and 
\cite[p.~702, 6.623.2]{gradshteyn-2007a}
\begin{equation}\label{jbesslap}
\int_{0}^{\infty}e^{-ax}J_{\nu}(bx)x^{\nu+1}\, dx=\frac{2a(2b)^{\nu}\Gamma\left(\nu+\tfrac{3}{2}\right)}{\sqrt{\pi}(a^2+b^2)^{\nu+\tfrac{3}{2}}},
\end{equation}
for Re $\nu>-1$ and Re $a>|$ Im $b|$. Substitute 
(\ref{psilog}) on the extreme right of (\ref{integ3}), and then use (\ref{jbesslap}) with $\nu=0, a=2\pi y$ and $b=2\pi t\alpha$ in the resulting equation 
to see (after simplification) that
\begin{align}\label{integ4}
H(\alpha)&=\frac{\sqrt{\alpha}}{16\pi^2}\int_{0}^{\infty}\left(\frac{e^{-2\alpha t}}{t}+\frac{4\pi\alpha t}{e^{2\pi t}-1}\int_{0}^{\infty}\frac{y}{(y^2+(\alpha t)^2)^{3/2}}\left(\frac{1}{e^{2\pi y}-1}-\frac{1}{2\pi y}\right)\, dy\right)\, dt\nonumber\\
&=\frac{\sqrt{\alpha}}{16\pi^2}\int_{0}^{\infty}\left(\frac{e^{-2\alpha t}}{t}+\frac{4\pi t}{e^{2\pi t}-1}\int_{0}^{\infty}\frac{y}{(y^2+t^2)^{3/2}}\left(\frac{1}{e^{2\pi\alpha y}-1}-\frac{1}{2\pi\alpha y}\right)\, dy\right)\, dt
\end{align}
Finally, use \cite[p.~675, 6.554.4]{gradshteyn-2007a} 
\begin{equation}
\int_{0}^{\infty}\frac{yJ_{0}(2\alpha y)}{(y^2+t^2)^{3/2}}\, dy=\frac{e^{-2\alpha t}}{t}, \quad \text{ for } \alpha, \, t > 0,
\end{equation}
on the extreme right of (\ref{integ4}) to prove  \eqref{ramdit}.
\end{proof}

\textit{\textbf{Proof of Theorem \ref{thm-45}}}. The first proof uses 
Theorem \ref{production-1} with 
$F(s,z) = \Gamma \left( s + z/2 \right) \Gamma(1 - s + z/2)$. This requires 
the evaluation of 
\begin{multline}
\label{formula-313}
f(x,z) = \frac{1}{2 \pi i } \int_{c - i \infty}^{c + i \infty} 
\Gamma \left( s + \frac{z}{2} \right)
\Gamma \left( 1- s + \frac{z}{2} \right) \\
\times \zeta \left( 1- s - \frac{z}{2} \right)
\zeta \left( 1- s + \frac{z}{2} \right)
x^{-s} ds
\end{multline}
\noindent
for $\pm \realpart{\frac{z}{2}} < c = \realpart{s} < 
1 \pm \realpart{\frac{z}{2}}$. 

First assume $c=$ Re $s>1\pm$ Re $\frac{z}{2}$. Denote the right-hand side of \eqref{formula-313} by $I(x,z)$. The change of 
variables $s= 1 - w$, gives
\begin{equation}\label{ixw}
I(x,w) = \frac{1}{2 \pi i} \int_{\lambda - i \infty}^{\lambda + i \infty}
H(x,w,z) dw,
\end{equation}
\noindent
where $\lambda = \realpart{w} < \pm \realpart{\frac{z}{2}}$ and 
\begin{equation}
H(x,w, z) = 
\Gamma \left( w + \frac{z}{2} \right) 
\Gamma \left( 1 - w + \frac{z}{2} \right) 
\zeta \left( w - \frac{z}{2} \right) 
\zeta \left( w + \frac{z}{2} \right) 
x^{w -1}.
\end{equation}
The line of integration is now shifted from $\realpart{w} = \lambda 
< \pm \realpart{\frac{z}{2}}$ to $\realpart{w} = \lambda' > 
1 \pm \realpart{\frac{z}{2}}$. These two lines are closed to form the 
rectangular contour with sides $( \lambda - iT, \lambda' - iT), \, 
( \lambda' - iT, \lambda' + iT), \, 
( \lambda' + iT, \lambda + iT), \, 
( \lambda + iT, \lambda - iT)$, where $T > 0$. This shift 
encounters poles at $w = -z/2, \, 1 + z/2$ and 
$1 - z/2$ of orders $1, \, 2$ and $1$, respectively. The residue theorem
gives
\begin{multline}
\int_{\lambda - iT}^{\lambda + iT} H(x,w,z) dw = 
\left[ 
\int_{\lambda - iT}^{\lambda'-  iT} +
\int_{\lambda' - iT}^{\lambda'+  iT} +
\int_{\lambda' + iT}^{\lambda +  iT} 
\right] H(x,w,z) dw  \\
- 2 \pi i \left[ R(-z/2) + R(1+z/2) + R(1-z/2) \right]
\end{multline}
\noindent
where $R(a)$ denotes the residue at the pole $a$. It is easy to see that 
the integrals along the horizontal segments tend to $0$ as $T \to \infty$. 
Therefore 
\begin{multline}
\int_{\lambda - i \infty}^{\lambda + i \infty} H(x,w,z) dw = 
\int_{\lambda' - i \infty}^{\lambda'+  i \infty} 
H(x,w,z) dw  \\
- 2 \pi i \left[ R(-z/2) + R(1+z/2) + R(1-z/2) \right].
\label{int-a1}
\end{multline}
The line integral is evaluated using \eqref{expansion-sigma} to obtain 
{\allowdisplaybreaks\begin{multline}
\label{formula-319}
\int_{\lambda' - i \infty}^{\lambda' + i \infty} H(x,w,z) dw =  \\
\frac{1}{x} \sum_{n=1}^{\infty} \sigma_{-z}(n) n^{z/2} 
\int_{\lambda' - i \infty}^{\lambda' + i \infty} 
\Gamma \left( w + \frac{z}{2} \right)
\Gamma \left( 1- w + \frac{z}{2} \right)
\left( \frac{n}{x} \right)^{-w} dw,
\end{multline}}
\noindent
for $\realpart{ (w \pm z/2)} > 1$. Now, for 
$ 0 < w_{0} = \realpart{w} < \realpart{z}$, 
\begin{equation}
\frac{1}{2 \pi i} \int_{w_{0} - i \infty}^{w_{0} + i \infty} 
\frac{\Gamma(w) \Gamma(z-w)}{\Gamma(z)} x^{-w} dw = 
\frac{1}{(1+x)^{z}},
\end{equation}
which, upon replacement of $w$ by $w+\frac{z}{2}$ and $z$ by $1+z$, gives
\noindent
\begin{equation}\label{melfs}
\frac{1}{2\pi i}\int_{d-i\infty}^{d+i\infty}\Gamma\left(w+\frac{z}{2}\right)\Gamma\left(1-w+\frac{z}{2}\right)x^{-w}\, dw=\frac{x^{\frac{z}{2}}\Gamma(1+z)}{(1+x)^{1+z}}
\end{equation}
for $- \realpart{\frac{z}{2}} < d = \realpart{w} < 1 + \realpart{\frac{z}{2}}$. Another application of the residue theorem leads to
\begin{multline}
\label{formula-322}
\frac{1}{2 \pi i} \int_{d - i \infty}^{d + i \infty} 
\Gamma \left(w + \frac{z}{2} \right) 
\Gamma \left( 1 - w + \frac{z}{2} \right)
x^{-w} dw = \\
\Gamma(1+z) \left[ \frac{x^{z/2}}{(1+x)^{1+z}} - x^{-1-z/2} \right]
\end{multline}
\noindent
for $d=\realpart{w}>1+\realpart{\frac{z}{2}}$.  Then \eqref{formula-319} and \eqref{formula-322} give
\begin{equation}
\frac{1}{2 \pi i } 
\int_{\lambda' - i \infty}^{\lambda' + i \infty} H(x,w,z) dw = 
x^{z/2} \Gamma(1+z) \sum_{n=1}^{\infty} \sigma_{-z}(n) 
\left( \frac{n^{z}}{(n+x)^{z+1}} - \frac{1}{n} \right),
\end{equation}
and \eqref{int-a1} gives
\begin{multline}
\label{int-a2}
\frac{1}{2 \pi i } 
\int_{\lambda - i \infty}^{\lambda + i \infty} H(x,w,z) dw = 
x^{z/2} \Gamma(1+z) \sum_{n=1}^{\infty} \sigma_{-z}(n) 
\left( \frac{n^{z}}{(n+x)^{z+1}} - \frac{1}{n} \right) \\
- R(-z/2) - R(1 + z/2) - R(1-z/2). 
\end{multline}

The computation of the residues at the poles yields the values
\begin{eqnarray}\label{residues}
\\ 
R(-z/2) 
& = & - \tfrac{1}{2} \Gamma(1+z) \zeta(-z)x^{-1-z/2} \nonumber\\
R(1+ z/2) & = & -x^{z/2} \Gamma(1+z) \left[ (2 \gamma + \log x + 
\psi(1+z)) \zeta(1+z) + \zeta'(1+z) \right] \nonumber\\
R(1-z/2) & = & x^{-z/2} \Gamma(z) \zeta(1-z). \nonumber 
\end{eqnarray} 
\noindent
For example, 
\begin{eqnarray*}
R(-z/2)& = &  \lim\limits_{w \to -\tfrac{z}{2}} H(x,w,z) \\
& = &  \lim\limits_{w \to -\tfrac{z}{2}} \left(w + \frac{z}{2} \right) 
\Gamma\left(w + \frac{z}{2} \right)
\Gamma \left(1-w+\frac{z}{2} \right) 
\zeta \left(w -\frac{z}{2} \right) 
\zeta \left(w+\frac{z}{2} \right)x^{w-1} \\
& = &  \lim\limits_{w \to - \tfrac{z}{2}} 
\Gamma \left(w + \frac{z}{2} + 1 \right) 
\Gamma \left(1-w+ \frac{z}{2} \right) 
\zeta \left(w - \frac{z}{2} \right) \zeta \left(w+ \frac{z}{2} \right)
x^{w-1}  \\
& = & -\frac{1}{2} \Gamma(1+z)\zeta(-z) x^{-z/2-1},
\end{eqnarray*}
\noindent
using $\Gamma(1)=1$ and $\zeta(0) = -1/2$. 

Now \eqref{lambda-def}, \eqref{ixw}, \eqref{int-a2} and \eqref{residues} give for $c=$ Re $s>1\pm$ Re $\frac{z}{2}$,
\begin{equation}
I(x,z)=\Lambda(x,z)+\frac{1}{2}\Gamma(1+z)\zeta(-z)x^{-1-\frac{z}{2}}.
\end{equation}
Finally, if $\pm$ \textup{Re} $\frac{z}{2}<c=$ \textup{Re} $s<1\pm $ \textup{Re} $\frac{z}{2}$, the residue theorem again produces
\begin{equation}
I(x, z)=\Lambda(x, z).
\end{equation}
Then,  \eqref{formula-313} implies $f(x, z)=\Lambda(x, z)$. Corollary 
\ref{cor5.4} now establishes the first equality in \eqref{formula-121}.

\smallskip

\noindent
\textbf{Second proof:} An alternative proof of the first equality begins with the integral on the right-hand side of 
\eqref{formula-121} written in the form
\begin{multline}
\label{formula-326}
I(z; \alpha)  =  \int_{0}^{\infty} 
\Gamma \left( \frac{z-1+ it}{4} \right) 
\Gamma \left( \frac{z-1- it}{4} \right) 
\Gamma \left( \frac{z+1+ it}{4} \right)  \\
\times \Gamma \left( \frac{z+1- it}{4} \right) 
\Xi \left( \frac{t + iz}{2} \right) 
\Xi \left( \frac{t - iz}{2} \right) 
\frac{\cos( \tfrac{1}{2} t \log \alpha)}{(z+1)^{2} + t^{2}} \, dt.
\end{multline}
 
Let 
\begin{equation}
f(z,t) = \frac{1}{(z+1)^{2} + 4t^{2}} 
\Gamma \left( \frac{z-1}{4} + \frac{it}{2} \right)
\Gamma \left( \frac{z-1}{4} - \frac{it}{2} \right)
\Gamma \left( \frac{z+1}{4} + \frac{it}{2} \right)
\Gamma \left( \frac{z+1}{4} - \frac{it}{2} \right)
\end{equation}
\noindent
that admits a factorization of the type \eqref{factorization-0} with 
\begin{equation}
\phi(z,s) = \frac{1}{1+z+2s} 
\Gamma \left( \frac{z-1}{4} + \frac{s}{2} \right)
\Gamma \left( \frac{z+1}{4} + \frac{s}{2} \right).
\end{equation}
\noindent
Using \eqref{formula-210}, 
\begin{eqnarray*}
I(z,\alpha) & = & \frac{1}{i \sqrt{\alpha}} 
\int_{\tfrac{1}{2} - i \infty}^{\tfrac{1}{2} + i \infty} 
\frac{1}{(z+2s)(z+2-2s)}  \\
& & \times \Gamma \left( \frac{z}{4} + \frac{s-1}{2} \right)
\Gamma \left( \frac{z}{4} + \frac{s}{2} \right)
\Gamma \left( \frac{z}{4} - \frac{s}{2} \right)
\Gamma \left( \frac{z}{4} + \frac{1-s}{2} \right) \nonumber \\
& & \times \frac{1}{2} \left(s - \frac{z}{2}\right) \left(s - \frac{z}{2} -1 \right) 
\pi^{-(s-z/2)/2} 
\Gamma \left( \frac{s}{2} - \frac{z}{4} \right)
\zeta \left( s - \frac{z}{2} \right) \nonumber \\
& & \times \frac{1}{2} \left(s + \frac{z}{2}\right) \left(s + \frac{z}{2} -1 \right) 
\pi^{-(s+z/2)/2} 
\Gamma \left( \frac{s}{2} + \frac{z}{4} \right)
\zeta \left( s + \frac{z}{2} \right)  \alpha^{s} \, ds.
\nonumber
\end{eqnarray*}
\noindent 
Simplifying the integrand, this is written as 
\begin{equation}
\label{int-forG}
I(z,\alpha) = \frac{\pi}{2^{z+1}i \sqrt{\alpha}} 
\int_{\tfrac{1}{2} - i \infty}^{\tfrac{1}{2} + i \infty} 
G(\alpha, s, z) ds,
\end{equation}
\noindent 
with
\begin{multline}
\label{formula-329} 
G(\alpha,s,z) = 
\Gamma \left( s + \frac{z}{2} \right)
\Gamma \left( \frac{z}{2} -s + 1 \right)
\Gamma \left( \frac{s}{2} - \frac{z}{4} \right)
\Gamma \left( \frac{s}{2} + \frac{z}{4} \right) \\
\times \zeta \left( s - \frac{z}{2} \right)
\zeta \left( s + \frac{z}{2} \right)
\left( \frac{\pi}{\alpha} \right)^{-s}.
\end{multline}

To evaluate this last integral one employs \eqref{expansion-sigma}. As 
before, this requires to move the line of integration from $\realpart{s} = 
\tfrac{1}{2}$ to $\realpart{s} = \tfrac{3}{2}$. In this 
shift, one encounters a simple 
pole  at $s = 1 - z/2$ and a double pole at $s = 1 + z/2$. The residue 
theorem produces 
\begin{equation}\label{resG}
\int_{\tfrac{1}{2} - i \infty}^{\tfrac{1}{2} + i \infty} 
G(\alpha, s, z) ds = 
\int_{\tfrac{3}{2} - i \infty}^{\tfrac{3}{2} + i \infty} 
G(\alpha, s, z) ds  - 2 \pi i \left( R(1- z/2) + R(1 +z/2) \right),
\end{equation}
\noindent
where $R(a)$ is the residue of $G(\alpha,s,z)$ at the pole $s= a$. A direct 
computation shows that 
\begin{equation}\label{r1}
R( 1 - z/2) = \frac{1}{z} \alpha^{1-z/2} \pi^{-z/2} \Gamma \left( \frac{z}{2} 
\right) \Gamma(z+1) \zeta(z)
\end{equation}
\noindent
and
\begin{align}
\label{r2}
R( 1 + z/2)  & = - \frac{\alpha^{1+z/2}}{2 \pi^{\tfrac{1}{2}(z+1)}} 
\Gamma \left( \frac{z+1}{2} \right) \Gamma(z+1)  \\
&\quad\times  
\left[ \left( 3 \gamma - 2 \log(2 \pi /\alpha) + \psi 
\left( \tfrac{1}{2}(z+1) \right) + 2 \psi(z+1) \right) \zeta(z+1) +
2 \zeta'(z+1) \right].
\nonumber
\end{align}

The expansion \eqref{expansion-sigma} gives
\begin{multline*}
\int_{\tfrac{3}{2} - i \infty}^{\tfrac{3}{2} + i \infty} 
G(\alpha, s, z) ds   = \\
\sum_{n=1}^{\infty} \sigma_{-z}(n) n^{z/2} 
\int_{\tfrac{3}{2} - i \infty}^{\tfrac{3}{2} + i \infty} 
\Gamma \left( s + \frac{z}{2} \right)
\Gamma \left(  \frac{z}{2} -s + 1\right)
\Gamma \left(  \frac{s}{2} - \frac{z}{4} \right)
\Gamma \left(  \frac{s}{2} + \frac{z}{4} \right)
\left( \frac{\pi n}{\alpha} \right)^{-s} ds.
\end{multline*}

In order to evaluate the integral on the right-hand side, one would like to 
use 
\eqref{parseval-2} with 
\begin{equation}\label{FG}
F(s) = \Gamma \left( s + \frac{z}{2} \right) 
\Gamma \left( \frac{z}{2} - s + 1 \right) \text{ and }
G(s) = \Gamma \left( \frac{s}{2} - \frac{z}{4} \right) 
\Gamma \left( \frac{s}{2} + \frac{z}{4} \right). 
\end{equation}
\noindent
Now \eqref{mellin-K} yields 
\begin{equation}\label{mellin-Kz2}
\frac{1}{2 \pi i} \int_{c - i \infty}^{c + i \infty} 
\Gamma \left( \frac{s}{2} - \frac{z}{4} \right) 
\Gamma \left( \frac{s}{2} + \frac{z}{4} \right) x^{-s} ds = 4 K_{z/2}(2x),
\end{equation}
\noindent 
for $c = \realpart{s} > \pm \realpart{z/2}$ and this is true when $c=\frac{3}{2}$. However, (\ref{melfs}) (with $w$ replaced by $s$) holds only for $-$ Re $\frac{z}{2}<d=$ Re $s<1+$ Re $\frac{z}{2}$, which is not satisfied when Re $s=\frac{3}{2}$ and $-1<$ Re $z<1$. Thus, the line of integration has to be moved from $\realpart{s} = \tfrac{3}{2}$ to 
$\realpart{s} = \tfrac{1}{2}$. This process captures a pole at $s=1+\frac{z}{2}$ and the residue 
theorem yields 
\begin{multline}
\int_{\tfrac{3}{2} - i \infty}^{\tfrac{3}{2} + i \infty} 
\Gamma \left( s + \frac{z}{2} \right)
\Gamma \left(  \frac{z}{2} -s + 1\right)
\Gamma \left(  \frac{s}{2} - \frac{z}{4} \right)
\Gamma \left(  \frac{s}{2} + \frac{z}{4} \right)
\left( \frac{\pi n}{\alpha} \right)^{-s} ds  \label{formula-338} \\
 = \int_{\tfrac{1}{2} - i \infty}^{\tfrac{1}{2} + i \infty} 
\Gamma \left( s + \frac{z}{2} \right)
\Gamma \left(  \frac{z}{2} -s + 1\right)
\Gamma \left(  \frac{s}{2} - \frac{z}{4} \right)
\Gamma \left(  \frac{s}{2} + \frac{z}{4} \right)
\left( \frac{\pi n}{\alpha} \right)^{-s} ds \nonumber \\
\quad - 2 \pi i \Gamma(z+1) \Gamma \left( \frac{z+1}{2} \right) 
\frac{\alpha^{1 + z/2}}{\pi^{\tfrac{z+1}{2}} n^{ 1 + \tfrac{z}{2}}}.
\nonumber
\end{multline}
\noindent
Now \eqref{parseval-2}, \eqref{melfs}, \eqref{FG} and \eqref{mellin-Kz2} yield
\begin{multline}
\int_{\tfrac{1}{2} - i \infty}^{\tfrac{1}{2} + i \infty} 
\Gamma \left( s + \frac{z}{2} \right)
\Gamma \left(  \frac{z}{2} -s + 1\right)
\Gamma \left(  \frac{s}{2} - \frac{z}{4} \right)
\Gamma \left(  \frac{s}{2} + \frac{z}{4} \right)
\left( \frac{\pi n}{\alpha} \right)^{-s} ds \\
= 4 \Gamma(z+1) \int_{0}^{\infty} \frac{x^{z/2} K_{z/2} \left( 
\frac{2 \pi n x}{\alpha} \right) }{(1+x)^{z+1}} \, dx
\end{multline}
\noindent 
and so this leads to 
\begin{multline}
\int_{\tfrac{3}{2} - i \infty}^{\tfrac{3}{2} + i \infty} 
G( \alpha, s ,z ) ds = \\
8 \pi i \Gamma(z+1) 
\int_{0}^{\infty} x^{z/2} K_{z/2} \left( \frac{2 \pi x}{\alpha} \right) 
\sum_{n=1}^{\infty} 
\sigma_{-z}(n) \left( \frac{n^{z}}{(n+x)^{z+1}} - \frac{1}{n} \right),
\end{multline}
where we used the integral representation
\begin{equation*}
\Gamma\left(\frac{z+1}{2}\right)\frac{\alpha^{1+\frac{z}{2}}}{\pi^{\frac{z+1}{2}}n^{1+\frac{z}{2}}}=4\int_{0}^{\infty}x^{z/2}K_{\frac{z}{2}}\left(\frac{2\pi nx}{\alpha}\right)\, dx
\end{equation*}
has been used. Using 
\eqref{resG}, the integral \eqref{int-forG} now becomes 
\begin{multline}\label{izalpha} 
I(z, \alpha) = \frac{\pi^{2}}{2^{z} \sqrt{\alpha}} 
\left\{ -\left( R \left(1-\frac{z}{2} \right) + 
R \left(1+ \frac{z}{2} \right) \right) \right. \\
\left.  \quad + 4 \Gamma(z+1) 
\int_{0}^{\infty} x^{z/2} K_{z/2} \left( \frac{2 \pi x}{\alpha} \right) 
\sum_{n=1}^{\infty} 
\sigma_{-z}(n) \left( \frac{n^{z}}{(n+x)^{z+1}} - \frac{1}{n} \right)
\right\}.
\end{multline}

The next step in the proof is to obtain an integral representation for the sum of the residues in the form 
\begin{equation*}
\int_{0}^{\infty} x^{z/2} K_{z/2}(2 \pi \alpha x) \lambda(x,z) dx 
\end{equation*}
\noindent
for an appropriate function $\lambda(x,z)$. 

The choice $w = 1 + z/2, \, \nu = z/2$ and $m = 2 \pi /\alpha$ in \eqref{mkl} gives 
\begin{multline}\label{r3}
\int_{0}^{\infty} 4 x^{z/2} K_{z/2}\left( \frac{2 \pi x}{\alpha} \right)
 \log x \, dx =  \\
\frac{\alpha^{1 + z/2}}{2 \pi^{\tfrac{z+1}{2}}} 
\Gamma \left( \frac{z + 1}{2} \right)
\left[ -\gamma 
-2 \log \left( \frac{2 \pi}{\alpha} \right) +
\psi \left( \frac{z+1}{2} \right) \right]. 
\end{multline}
\noindent
Now \eqref{mellin-K} produces 
\begin{equation}\label{r4}
\int_{0}^{\infty} 4 x^{z/2} K_{z/2} \left( \frac{2 \pi x}{\alpha} \right) 
\frac{x^{-z}}{-z} \zeta(1-z) dx = 
-\frac{1}{z} \pi^{-z/2} \alpha^{1-z/2} \Gamma \left( \frac{z}{2} \right) 
\zeta(z),
\end{equation}
\noindent
and also
\begin{multline}\label{r5}
\int_{0}^{\infty} 4x^{z/2} K_{z/2} \left( \frac{2 \pi x}{\alpha} \right) 
\left[ \left(2 \gamma + \psi(z+1) \right) \zeta(z+1) + \zeta'(z+1) \right] dx \\
= \frac{\alpha^{1+z/2}}{\pi^{\tfrac{z+1}{2}}}
\Gamma \left( \frac{z+1}{2} \right)
\left[ 
\left(2 \gamma + \psi(z+1) \right) \zeta(z+1) + \zeta'(z+1) \right].
\end{multline}
Now \eqref{r1}, \eqref{r2}, \eqref{r3}, \eqref{r4} and \eqref{r5} yield
\begin{multline}\label{sumres}
R(1-\frac{z}{2})+R(1+\frac{z}{2})\\
=-\Gamma(z+1)\int_{0}^{\infty}4x^{\frac{z}{2}}K_{\frac{z}{2}}\left(\frac{2\pi x}{\alpha}\right)\bigg\{\frac{x^{-z}}{-z}\zeta(1-z)+(2\gamma+\log x+\psi(z+1))\zeta(z+1)+\zeta'(z+1)\bigg\}\, dx.
\end{multline}
Substituting \eqref{sumres} in \eqref{izalpha} gives the equality between the extreme left and right sides of \eqref{formula-121}, and the invariance of the right side under $\alpha\to 1/\alpha$ establishes the other as well.

\smallskip

\noindent
\textit{\textbf{Proof of Theorem \ref{fergenthm}}}. Since the 
proof of Theorem \ref{fergenthm} is similar to that of Theorems \ref{thm-45} 
and \ref{thm-46}, a brief outline is presented. Let 
\begin{equation}
F(s, z)=\Gamma\left(\frac{s}{2}+\frac{z}{4}\right)\Gamma\left(\frac{1}{2}-\frac{s}{2}+\frac{z}{4}\right),
\end{equation}
apply Theorem \ref{production-1} to find by means of contour integration that $f(x, z)=\mathfrak{F}(x, z)$, as in \eqref{frakf}. Then as before, Corollary \ref{cor5.4} gives the transformation in (\ref{fergen}). Also, starting from the integral on the extreme right of (\ref{fergen}), converting it into a complex integral using (\ref{formula-210}), and then using the residue theorem and Mellin transforms, one finds this integral to be equal to one of the two sides of the transformation in \eqref{fergen}. The transformation itself is then obtained by replacing $\alpha$ by $\beta$ in this integral involving the Riemann $\Xi$ function.

\bigskip

\section{A generalization of a series of Koshlyakov}
\label{sec-omega}

Koshlyakov \cite{koshlyakov-1929a} considered the function 
\begin{equation}
\Omega(x) = 2 \sum_{n=1}^{\infty} d(n) \left[ K_{0} \left( 4 \pi e^{i \pi/4} 
\sqrt{nx} \right) + 
K_{0} \left( 4 \pi e^{-i \pi/4} 
\sqrt{nx} \right) \right]\label{koshomega}
\end{equation}
and used it to give a short and clever proof of the Vorono\"{\dotlessi} summation formula.
\noindent
In \cite[Equation 5]{koshlyakov-1929a}, he established the identity 
\begin{equation}
\label{omega-11}
\Omega(x) = - \gamma - \frac{1}{2} \log x - \frac{1}{4 \pi x} +
\frac{x}{\pi} \sum_{n=1}^{\infty} \frac{d(n)}{x^{2} + n^{2}}. 
\end{equation}

The reader will find in \cite{dixit-2013g} how to establish \eqref{omega-11} 
using \eqref{lostomega}.
Koshlyakov \cite[Equation (6)]{koshlyakov-1936b}, 
\cite[Equations (21), (27)]{koshlyakov-1937a} showed that
\begin{align}\label{koshram}
\sqrt{\alpha}\int_{0}^{\infty}e^{-2\pi\alpha x}\left(\Omega(x)+\frac{1}{4\pi x}\right)\, dx &= \sqrt{\beta}\int_{0}^{\infty}e^{-2\pi\beta x}\left(\Omega(x)+\frac{1}{4\pi x}\right)\, dx\nonumber\\
&=\dfrac{1}{2\pi^{5/2}}\int_0^{\infty}\left|\Xi\left(\dfrac{1}{2}t\right)\Gamma\left(\dfrac{-1+it}{4}\right)\right|^2
\dfrac{\cos\left(\tfrac{1}{2}t\log\alpha\right)}{1+t^2}\, dt.
\end{align}
See \cite{dixit-2013g} for the proof of the equivalence 
of the formula in Theorem \ref{thm-3.4}
and (\ref{koshram}) without appealing to the integral involving the Riemann $\Xi$-function in the 
identities.  Koshlyakov \cite[Equations (8), (24)]{koshlyakov-1934a}, 
\cite[Equation (15)]{koshlyakov-1936a} also 
found another transformation involving $\Omega(x)$, namely, 
\begin{align}\label{jbessel}
&\sqrt{\alpha^3}\int_{0}^{\infty}xJ_{0}(2\pi\alpha x)\left(\Omega(x)+\frac{1}{4\pi x}\right)\, dx=\sqrt{\beta^3}\int_{0}^{\infty}xJ_{0}(2\pi\beta x)\left(\Omega(x)+\frac{1}{4\pi x}\right)\, dx\nonumber\\
&=\frac{1}{64\pi^{5}}\int_{0}^{\infty}\left|\Gamma^2\left(\frac{-1+it}{4}\right)\right|^2\Xi^{2}\left(\frac{t}{2}\right)\cosh\left(\frac{1}{2}\pi t\right)\cos\left(\tfrac{1}{2}t\log\alpha\right)\, dt.
\end{align}
A new generalization of $\Omega(x)$, defined by
\begin{equation}
\Omega(x,z) = 2 \sum_{n=1}^{\infty} \sigma_{-z}(n) n^{z/2} 
\left( e^{\pi i z/4} K_{z}( 4 \pi e^{\pi i/4} \sqrt{nx} ) +
 e^{-\pi i z/4} K_{z}( 4 \pi e^{-\pi i/4} \sqrt{nx} ) \right),
\end{equation}
\noindent
is considered next. It 
is clear that $\Omega(x, 0)=\Omega(x)$. The inverse Mellin transform
\begin{equation}
\label{Mellin-omega}
\Omega(x,z) = \frac{1}{2 \pi i} 
\int_{c - i \infty} ^{c + i \infty} 
\frac{\zeta(1 - s + \tfrac{z}{2}) \zeta(1 -s  - \tfrac{z}{2} ) }
{2 \cos\left( \tfrac{1}{2} \pi \left( s + \tfrac{z}{2} \right)\right)} x^{-s} ds
\end{equation}
is validd for $c=$ Re $s>1\pm$ Re $\frac{z}{2}$. The 
special case $z=0$ was given by Koshlyakov 
\cite[Equation (11)]{koshlyakov-1936a} and the details for deriving the general case are similar to this special case.

The function $\Omega(x, z)$ plays an important role in deriving a 
simpler proof of the generalization
of the Vorono\"{\dotlessi} summation formula. See \cite{berndt-2013a} for 
details.

\begin{proposition}\label{propomega}
For $z\notin\mathbb{Z}$, the function $\Omega(x,z)$ is given by 
\begin{equation}
\Omega(x,z) = - \frac{\Gamma(z) \zeta(z)}{(2 \pi \sqrt{x})^{z}}  +
\frac{x^{z/2-1}}{2 \pi} \zeta(z) - 
\frac{x^{z/2}}{2} \zeta(z+1)  +
\frac{x^{z/2+1}}{\pi} \sum_{n=1}^{\infty} \frac{\sigma_{-z}(n)}{n^{2}+x^{2}}.
\end{equation}
\end{proposition}
\noindent
Actually, the above result is also true for $z=0$ (in the limiting sense) in view of (\ref{omega-11}) and the fact
that $\Omega(x, z)$ is continuous at $z=0$.

The result in the proposition can be directly derived from the next theorem of 
H. Cohen \cite{cohen-2010a}. Simply take $k=1$ and replace $x$ by 
$ix$ and then by $-ix$ and add the results.\\

\begin{theorem}[H. Cohen]
Let $z \not \in \mathbb{Z}$ be such that $\realpart{z} \geq 0$. For any 
integer $k \geq \lf \tfrac{1}{2}(\realpart{z} + 1 ) \rf$, 
\begin{multline*}
8 \pi x^{z/2} \sum_{n=1}^{\infty} 
\sigma_{-z}(n) n^{z/2} K_{z}(4 \pi \sqrt{nx}) =
A(z,x) \zeta(z) + B(z,x) \zeta(z+1) \\
+ \frac{2}{\sin(\pi z/2)} 
\left( \sum_{1 \leq j \leq k} \zeta(2j) \zeta(2j-z) x^{2j-1} +
x^{2k+1} \sum_{n=1}^{\infty} \sigma_{-z}(n) 
\frac{n^{z-2k} - x^{z-2k}}{n^{2}-x^{2}} \right),
\end{multline*}
\noindent
where
\begin{equation*}
A(z,x) = \frac{x^{z-1}}{\sin(\pi z/2)} - \frac{\Gamma(z)}{(2 \pi)^{z-1}} 
\text{ and }
B(z,x) = \frac{2}{x}\frac{\Gamma(z+1)}{(2 \pi)^{z+1}} - 
\frac{\pi x^{z}}{\cos(\pi z/2)}.
\end{equation*}
\end{theorem}

\medskip

A generalization of \eqref{koshram} is stated next. 
\begin{theorem}\label{thm-ram00}
Assume $-1 < \realpart{z} < 1$. Then for $\alpha, \beta>0, \alpha\beta=1$, 
\begin{multline*}
\alpha^{(z+1)/2} 
\int_{0}^{\infty} e^{-2 \pi \alpha x} x^{z/2} 
\left( \Omega(x,z) - \frac{1}{2 \pi} \zeta(z) x^{z/2-1} \right) dx \\
=\beta^{(z+1)/2} 
\int_{0}^{\infty} e^{-2 \pi \beta x} x^{z/2} 
\left( \Omega(x,z) - \frac{1}{2 \pi} \zeta(z) x^{z/2-1} \right) dx\\
=\frac{8}{\pi^{(z+5)/2}}
\int_{0}^{\infty} 
\Gamma \left( \frac{z+3+it}{4} \right)
\Gamma \left( \frac{z+3-it}{4} \right)
\Xi \left( \frac{t + iz}{2} \right)
\Xi \left( \frac{t - iz}{2} \right) \\
\times \frac{\cos( \tfrac{1}{2} t \log \alpha) \, dt}{(t^{2} + (z+1)^{2})
(t^{2} + (z-1)^{2})}.
\end{multline*}
\end{theorem}
\begin{proof}
The theorem is proven first for $0<$ Re $z<1$ and later extend it to 
$-1<$ Re $z<1$ by analytic continuation. Using (\ref{formula-210}) 
with $\phi(z, s)=\frac{1}{2s+z+1}\Gamma\left(\frac{z-1}{4}+\frac{s}{2}\right)$ and the reflection and the duplication formulas for the gamma function, the integral on the extreme right above (say $M(z, \alpha)$) can be written as
\begin{align}\label{krg1}
M(z,\alpha):=\frac{\sqrt{\pi}}{2^{1+\frac{z}{2}}i\sqrt{\alpha}}\int_{\tfrac{1}{2}-i\infty}^{\tfrac{1}{2}+i\infty}\Gamma\left(1-s+\frac{z}{2}\right)\frac{\zeta\left(1-s+\frac{z}{2}\right)\zeta\left(1-s-\frac{z}{2}\right)}{2\cos\left(\frac{1}{2}\pi\left(s+\frac{z}{2}\right)\right)}\left(\frac{1}{2\pi\alpha}\right)^{-s}\, ds.
\end{align}
To use \eqref{parseval-2}, one needs to evaluate the inverse Mellin transforms of the two functions, namely $\zeta\left(1-s+\frac{z}{2}\right)\zeta\left(1-s-\frac{z}{2}\right)/\left(2\cos\left(\frac{1}{2}\pi\left(s+\frac{z}{2}\right)\right)\right)$ and $\Gamma\left(1-s+\frac{z}{2}\right)$ in a common region which includes the vertical line Re $s=\tfrac{1}{2}$. 

For $c= \realpart{s}<1\pm \realpart{\left(\frac{z}{2}\right)}$, using (\ref{Mellin-omega}) and invoking the residue theorem results in
\begin{align}\label{omegamellin1}
\frac{1}{2\pi i}\int_{c-i\infty}^{c+i\infty}\frac{\zeta\left(1-s+\frac{z}{2}\right)\zeta\left(1-s-\frac{z}{2}\right)}{2\cos\left(\frac{1}{2}\pi\left(s+\frac{z}{2}\right)\right)}x^{-s}\, ds=\Omega(x, z)-\frac{\zeta(z)}{2\pi}x^{\frac{z}{2}-1}.
\end{align}
Also for $c=$ Re $s<1+$ Re $\left(\frac{z}{2}\right)$,
\begin{equation}
\frac{1}{2\pi i}\int_{c-i\infty}^{c+i\infty}\Gamma\left(1-s+\frac{z}{2}\right)x^{-s}\, ds=e^{-\frac{1}{x}}x^{-1-\frac{z}{2}}.
\end{equation}
Since $0<$ Re $z<1$, shifting the line of integration to 
$c= \realpart{s}<1- \realpart{\left(\frac{z}{2}\right)}$ does not 
introduce a pole. Therefore the above formula is valid for 
$c= \realpart{s}<1\pm \realpart{\left(\frac{z}{2}\right)}$. Also, the 
choice $c = \tfrac{1}{2}$ is valid since $0< \realpart{z}<1$. Thus 
employing (\ref{parseval-2}), gives 
\begin{align}\label{krg2}
&\int_{\tfrac{1}{2}-i\infty}^{\tfrac{1}{2}+i\infty}\Gamma\left(1-s+\frac{z}{2}\right)\frac{\zeta\left(1-s+\frac{z}{2}\right)\zeta\left(1-s-\frac{z}{2}\right)}{2\cos\left(\frac{1}{2}\pi\left(s+\frac{z}{2}\right)\right)}\left(\frac{1}{2\pi\alpha}\right)^{-s}\, ds\nonumber\\
&=2\pi i\int_{0}^{\infty}e^{-2\pi\alpha x}(2\pi\alpha x)^{1+\frac{z}{2}}\left(\Omega(x,z)-\frac{\zeta(z)}{2\pi}x^{\frac{z}{2}-1}\right)\, \frac{dx}{x}.
\end{align}
Therefore (\ref{krg1}) and (\ref{krg2}) establish 
the equality between the extreme left and right sides of 
Theorem \ref{thm-ram00}.
Now using \eqref{strivert} and known bounds on the Riemann zeta function, it 
is easy to see that $M(z, \alpha)$ is analytic in $-1<$ Re $z<1$. Similarly, 
using Proposition \ref{propomega}, it is easy to see that the extreme left 
side of Theorem \ref{thm-ram00} is analytic in $-1<$ Re $z<1$. Thus by 
analytic continuation, the equality holds for $-1<$ Re $z<1$. Now, as 
usual, replace $\alpha$ by $\beta$ in the established identity and use 
the relation $\alpha\beta=1$ to obtain the second equality 
in Theorem \ref{thm-ram00}.
\end{proof}
A generalization of \eqref{jbessel} is stated next. 
\begin{theorem}
\label{thm-ram01}
Assume $-1 < \realpart{z} < 1$. Then for $\alpha, \beta>0, \alpha\beta=1$,
\begin{multline*}
\sqrt{\alpha^{3}} \int_{0}^{\infty} x J_{z/2}( 2 \pi \alpha x) 
\left( \Omega(x,z) - \frac{1}{2 \pi} \zeta(z) x^{z/2-1} \right) \, dx\\
=\sqrt{\beta^{3}} \int_{0}^{\infty} x J_{z/2}( 2 \pi \beta x) 
\left( \Omega(x,z) - \frac{1}{2 \pi} \zeta(z) x^{z/2-1} \right) \, dx\\
=\frac{8}{\pi^{3}} 
\int_{0}^{\infty} \frac{
\Gamma \left( \frac{z+ 3 + it}{4} \right)
\Gamma \left( \frac{z+ 3 - it}{4} \right)
}{
\Gamma \left( \frac{z+ 1 + it}{4} \right)
\Gamma \left( \frac{z+ 1 - it}{4} \right)
} 
\Xi \left( \frac{t-iz}{2} \right)
\Xi \left( \frac{t+iz}{2} \right) \\
\times \frac{\cos( \tfrac{1}{2} t \log \alpha) \, dt}{(t^{2} + (z+1)^{2})
(t^{2} + (z-1)^{2})}.
\end{multline*}
\end{theorem}
\begin{proof}
The proof is based on the Mellin transform \eqref{Mellin-besselJ}, the inverse Mellin transform \eqref{omegamellin1} and \eqref{parseval-1}. Details 
are omitted.
\end{proof}

\begin{note}
The integral on the right-hand side of Theorem \ref{thm-ram00} appears in S.~Ramanujan 
\cite[Equation (20)]{ramanujan-1915a} and in \cite[Theorem 1.5]{dixit-2011b}, where alternate
representations for this integral have been given. Comparing the representation derived here with these, lead to the following identity:
\begin{multline}
\alpha^{(z+1)/2} \int_{0}^{\infty} e^{-2 \pi \alpha x} x^{z/2} 
\left( \Omega(x,z) - \frac{1}{2 \pi} \zeta(z) x^{z/2-1} \right) dx \\
 = \frac{1}{(2 \pi)^{z+1}} \int_{0}^{\infty} x^{z} 
\left( \frac{1}{e^{x \sqrt{\alpha}}-1} - \frac{1}{x \sqrt{\alpha}} \right)
\left( \frac{1}{e^{x/ \sqrt{\alpha}}-1} - \frac{1}{x/ \sqrt{\alpha}} \right)
dx \\
=  \alpha^{(z+1)/2} \frac{\Gamma(z+1)}{(2 \pi)^{z+1}} 
\left[ \sum_{n=1}^{\infty} \left( \zeta(z+1, n \alpha) - 
\frac{(n \alpha)^{-z}}{z} - \frac{(n \alpha)^{-z-1}}{2} \right) -
\frac{\zeta(z+1)}{2 \alpha^{z+1}} - \frac{\zeta(z)}{\alpha z} \right],
\end{multline}
\noindent
where $\zeta(z,x)$ is the Hurwitz zeta function. 
\end{note}

\medskip

\section{Related work of Guinand and of Nasim}
\label{sec-advantages}

The work presented here is related to results of 
Guinand and Nasim. These are presented next.  
Guinand \cite[Theorem 6]{guinand-1939a} (see 
also \cite[Equation (1)]{guinand-1955a} obtained 
the following summation formula involving $\sigma_{s}(n)$:
\begin{align}\label{guisum}
&\sum_{n=1}^{\infty}\sigma_{-s}(n)n^{\frac{s}{2}}f(n)-\zeta(1+s)\int_{0}^{\infty}x^{\frac{s}{2}}f(x)\, dx-\zeta(1-s)\int_{0}^{\infty}x^{-\frac{s}{2}}f(x)\, dx\nonumber\\
&=\sum_{n=1}^{\infty}\sigma_{-s}(n)n^{\frac{s}{2}}g(n)-\zeta(1+s)\int_{0}^{\infty}x^{\frac{s}{2}}g(x)\, dx-\zeta(1-s)\int_{0}^{\infty}x^{-\frac{s}{2}}g(x)\, dx.
\end{align}
Here $f(x)$ satisfies appropriate conditions 
(see \cite{guinand-1939a}  for details) and $g(x)$ is the 
transform of $f(x)$ with respect to the Fourier kernel
\begin{equation}\label{guifouker}
-2\pi\sin\left(\tfrac{1}{2}\pi s\right)J_{s}(4\pi\sqrt{x})-\cos\left(\tfrac{1}{2}\pi s\right)\left(2\pi Y_{s}(4\pi\sqrt{x})-4K_{s}(4\pi\sqrt{x})\right).
\end{equation}
Note that up to a constant factor, the above kernel is the same as the one used in \eqref{koshlyakov-1}. C.~Nasim \cite{nasim-1969a, nasim-1971a, nasim-1974a} also derived transformation formulas similar to \eqref{guisum}.

As an application of (\ref{guisum}), note that 
for $z$ fixed and $-1<$ Re $z<1$, one obtains \eqref{guinandf1} by taking $f(x)=K_{\frac{z}{2}}(2\pi\alpha x)$, and then using (\ref{koshlyakov-1}) with $\nu=z/2$ and (\ref{zetafe}). This 
is the simplest example of \eqref{guisum}, since here $f(x)=g(x)$. The 
disadvantage of (\ref{guisum}) is the difficulty in obtaining other explicit 
examples. Note that Guinand \cite{guinand-1955a} does 
not give any particular example of \eqref{guisum}\footnote{Guinand's 
proof of \eqref{guinandf1} is different from the one given above.}. The 
production of a function $g$ requires the explicit evaluation of the integral 
\begin{equation*}
\int_{0}^{\infty}f(y)\left(-2\pi\sin\left(\tfrac{1}{2}\pi s\right)J_{s}(4\pi\sqrt{xy})-\cos\left(\tfrac{1}{2}\pi s\right)\left(2\pi Y_{s}(4\pi\sqrt{xy})-4K_{s}(4\pi\sqrt{xy})\right)\right)\, dy.
\end{equation*}
Instead, considering transformations between two integrals, as done here, one 
can construct a variety of explicit examples. These are shown in 
Theorems \ref{thm-45}, \ref{thm-46} and \ref{fergenthm}.
Also in order to explicitly find a transformation using (\ref{guisum}), one has to start with a proper $f$. A priori, it is hard to conceive of an $f$ which 
would lead to an elegant transformation. The advantage in our case is that 
the results are obtained by 
considering known transformation formulas in the literature. Thus, one does
not have to begin with an unnatural choice of $f$.

\section{Future developments}\label{sec-dev}

The results presented here deal with  modular-type transformations which 
result from squaring the functional equation of $\zeta(s)$, or equivalently, 
those, whose associated integrals involving the Riemann $\Xi$-functions 
have $\frac{\Xi^{2}(t/2)}{(1+t^2)^{2}}$ in their integrand. These can be 
extended by taking higher powers of the functional equation, or 
equivalently, by taking $\frac{\Xi^{m}(t/2)}{(1+t^2)^{m}}$ in the integrand 
of the associated integrals. The consequences of this extension are discussed 
in a particular example.

In view of Hardy's result (\ref{hardyf}) and Koshlyakov's result (\ref{int-kosh1}), it is also possible to evaluate the 
integral 
$$\int_{0}^{\infty}\left(\frac{\Xi(t/2)}{(1+t^2)}\right)^{m}\frac{\cos\left(\frac{1}{2}t\log\alpha\right)}{\cosh\left(\frac{1}{2}\pi t\right)}\, dt,$$
and obtain the corresponding modular-type transformation. Note that 
when $m=1$, the transformation involves the series 
$$-\psi(x+1)-\gamma=\sum_{n=1}^{\infty}\left(\frac{1}{x+n}-\frac{1}{n}\right)$$
and when $m=2$, the series 
$$\sum_{n=1}^{\infty}d(n)\left(\frac{1}{x+n}-\frac{1}{n}\right).$$
\noindent
Thus for $m>2$, the transformation would involve the series 
$$\sum_{n=1}^{\infty}d_{m}(n)\left(\frac{1}{x+n}-\frac{1}{n}\right),$$
where $d_m(n)$ denotes the number of ways of expressing $n$ as a product 
of $m$ factors in which an expression with the same factors but in a 
different order is counted as different. The Dirichlet series 
for $d_m(n)$ is given by
\begin{equation}
\zeta^{m}(s)=\sum_{n=1}^{\infty}\frac{d_m(n)}{n^s}.
\end{equation}
The series $\sum_{n=1}^{\infty}d_{m}(n)\left(\frac{1}{x+n}-\frac{1}{n}\right)$ as well as the residues associated with it come from evaluating the integral 
\begin{equation}
\frac{1}{2\pi i}\int_{\frac{3}{2}-i\infty}^{\frac{3}{2}+i\infty}\frac{\pi}{\sin\pi s}\zeta^{m}(1-s)x^{-s}\, ds.
\end{equation}
Let $f$ be the function against which the above expression is 
integrated in order to obtain the transformation. In the case $m=1$, 
$f(x)=e^{-\pi\alpha^2 x^2}$ whose Mellin transform is 
$\frac{1}{2}\pi^{-s/2}\alpha^{-s}\Gamma\left(\frac{s}{2}\right)$, and 
when $m=2$, $f(x)=K_{0}(2\pi\alpha x)$ whose Mellin transform is 
$\frac{1}{4}\pi^{-s}\alpha^{-s}\Gamma^{2}\left(\frac{s}{2}\right)$. Thus, for 
$m>2$, the Mellin transform of $f$ should be a multiple of 
$\Gamma^{m}\left(\frac{s}{2}\right)$. These 
general transformations will be discussed in a future publication.

\medskip

\noindent
\textbf{Acknowledgments}. The second author acknowledges the partial 
support of NSF-DMS 1112656. The first author is a post-doctoral fellow, funded 
in part by the same grant.

\end{document}